\documentclass[12pt]{article}
\usepackage[utf8]{inputenc}
\usepackage{times}
\usepackage[centertags]{amsmath}
\usepackage{amssymb}
\usepackage{amsthm}
\usepackage{enumitem}
\usepackage[utf8]{inputenc}
\usepackage[T1]{fontenc}
\usepackage[colorlinks=true,urlcolor=red,citecolor=red,linkcolor=red]{hyperref}
\usepackage[a4paper, top=5cm, bottom=5cm, left=3.5cm, right=2.5cm]{geometry}
\usepackage[italian,english]{babel}
\usepackage{float}
\usepackage{mathtools}
\usepackage{graphicx}
\usepackage{caption}
\usepackage{subcaption}
\usepackage{tikz}
\usepackage{tikz-3dplot}
\usepackage{qtree}
\usepackage{tikz-qtree}
\usepackage[lined,algonl,boxed]{algorithm2e}
\usepackage{pgfplots}
\pgfplotsset{compat=1.13}
\usetikzlibrary{trees,positioning,shapes}
\usetikzlibrary{patterns}
\usepackage{verbatim}
\newtheorem{theorem}{Theorem}[section]
\newtheorem{definition}[theorem]{Definition}
\newtheorem{prop}[theorem]{Proposition}
\newtheorem{lemma}[theorem]{Lemma}
\newtheorem{cor}[theorem]{Corollary}
\theoremstyle{definition}
\newtheorem{oss}[theorem]{Remark}

\theoremstyle{definition}
\newtheorem{ex}[theorem]{Example}

\newcommand{\bs}{\boldsymbol}

\newcommand{\al}{\boldsymbol{\alpha}}
\newcommand{\be}{\boldsymbol{\beta}}
\newcommand{\de}{\boldsymbol{\delta}}
\newcommand{\e}{\boldsymbol{e}}
\newcommand{\om}{\boldsymbol{\omega}}
\newcommand{\g}{\boldsymbol{\gamma}}
\newcommand{\te}{\boldsymbol{\theta}}
\newcommand{\he}{\boldsymbol{\eta}}
\newcommand{\ga}{\boldsymbol{\gamma}}

\newcommand{\Ap}{\mathrm{\bf Ap}}
\newcommand{\wt}{\hspace{0.1cm}\widetilde{\wedge}\hspace{0.1cm}}
\newcommand{\wh}[1]{\widehat{#1}}
\newcommand{\bei}[1]{\bs{\beta^{(#1)}}}
\newcommand{\td}[2]{\widetilde\Delta_{#1}^S(#2)}
\newcommand{\ds}[2]{\Delta_{#1}^S(#2)}
\newcommand{\dE}[2]{\Delta_{#1}^E(#2)}
\newcommand{\tdE}[2]{\widetilde\Delta_{#1}^E(#2)}

\newcommand{\N}{\mathbb{N}}

\newcommand{\K}{\mathbb{K}}
\newcommand{\pf}{\mathrm{\bf PF}}
{\null\vfill\begin{center}%
		\bfseries Acknowledgements\end{center}}%
{\vfill\null}
\newcommand{\keywords}[1]{\emph{Keywords:} #1}
\newcommand{\MSC}[1]{\emph{Mathematics Subject Classification 2010:} #1}

\newcounter{lastnote}

\begin{document}

\title{Properties and applications of the Ap\'ery set of good semigroups in $\mathbb{N}^d$} 

\author{L. Guerrieri \thanks{Jagiellonian University, Instytut Matematyki, 30-348 Krak\'{o}w \emph{e-mail: lorenzo.guerrieri@uj.edu.pl}},\quad N. Maugeri \thanks{Università degli studi di Catania, Dipartimento di Matematica e Informatica, Catania \emph{e-mail: nicola.maugeri@unict.it}},\quad V. Micale \thanks{Università degli studi di Catania, Dipartimento di Matematica e Informatica, Catania \emph{e-mail: vmicale@dmi.unict.it}
%\newline 
}
}

\date{}

\linespread{1,0}
\maketitle
\begin{abstract}
\noindent In this article we discuss some applications of the construction of the Ap\'ery set of a good semigroup in $\mathbb{N}^d$ given in \cite{GMM}. In particular we study: the duality of a symmetric and almost symmetric good semigroup, the Ap\'ery set of non-local good semigroups and the Ap\'ery set of value semigroups of plane curves.
\end{abstract}
\\
\MSC{20M10, 20M14, 20M25.}

%\begin{document}

\maketitle

\section{Introduction}
In this article we discuss several applications of the construction of the Ap\'ery set of a good semigroup given in \cite{GMM}. Good semigroups form a family of submonoids of $\mathbb{N}^d$ defined axiomatically in \cite{a-u}, in order to study Noetherian analytically unramified one-dimensional semilocal reduced rings, e.g. the local
rings arising from curve singularities and their blowups. 
Indeed, the family of good semigroups contains all value semigroups of such algebroid curves.  The concept of value semigroup has been already known long time before the definition of good semigroup was given, and, often in the literature, properties of algebroid curves and of the corresponding rings have been translated and studied at semigroup level \cite{kunz}, \cite{H-K} \cite{felix}, \cite{delgado}, \cite{C-D-K}, \cite{danna1}, \cite{c-d-gz}, \cite{two}, \cite{KST}, \cite{D-G-M-T}.
For instance, it is well-known that a one-dimensional analytically unramified local ring
 is Gorenstein if and only if its value semigroup is symmetric. 
The integer $d$ represents the number of branches of the curve. For $d=1$, good semigroups coincide with numerical semigroups and have been extensively studied in connection with many subjects including algebraic geometry, commutative algebra, factorization theory, coding theory \cite{libropedro}, \cite{libroassidannapedro}.
More recently, good semigroups in $\mathbb{N}^d$ with $d \geq 2$ have been studied, still in connection with the geometric and algebraic theory of curve singularities, but also with the purpose of extending pure combinatoric properties of numerical semigroups to this more general setting.

The concept of Ap\'ery set, classical notion in the theory of numerical semigroups, has been extended to the "good" case, first in \cite{apery} for value semigroups of plane curves with two branches, then for arbitrary good semigroups in $\mathbb{N}^2$ in \cite{DGM}, and for any good semigroup and any $d$ in \cite{GMM}.

This notion has been a fundamental tool to generalize various features of the numerical setting, obtaining new characterization of classes of good semigroups, such as symmetric and almost symmetric, and studying important invariants, such as type, embedding dimension, genus \cite{DGM}, \cite{DGM2}, \cite{NG}, \cite{NG2}.

Unfortunately, for non-numerical good semigroups, the Ap\'ery set is an infinite set, but it can be partitioned canonically in a finite number of subsets, called levels.  Properties of such levels reflect the behavior of the semigroup and have particularly nice applications.

In this paper we consider some of these various applications, with the idea of both extending results from the case $d \leq 2$ to arbitrary $d$ and also to cover complementary results which have not considered in the previous papers.  
Usually, instead of considering only the Ap\'ery sets, when possible we prove results for complements of good ideals (see definitions in Section 2), since this approach is more general and is needed in some of the cases we consider.

After a preliminary section (Section 2), in which we recall all the main definitions and results we are going to use, we consider in Section 3 the duality property of Ap\'ery sets of symmetric and almost symmetric good semigroups.
Such semigroups are interesting since when they are value semigroups of algebroid curves, they correspond to those algebroid curves having respectively Gorenstein and Almost Gorenstein ring (for references on the Almost Gorenstein case see \cite{B-Fr}, \cite{a-u}). We show how duality properties on the levels of the Ap\'ery set characterize these classes extending to any $d \geq 2$
the results obtained in \cite{DGM} and \cite{DGM2} for $d=2$ (and more classical results for $d=1$). 

In Section 4, we consider the complement of good ideals (and Ap\'ery sets) of non-local good semigroups (when they are value semigroups they correspond to non-local rings of algebroid curves). In this case a more precise description of the partition in levels can be obtain by splitting the semigroup as direct product of smaller good semigroups. This description will be also useful in Section 6.

In Section 5, we study a class of complements of good ideals, that we call \it well-behaved \rm and that includes the Ap\'ery sets of plane curves. Also in this case, we can give a better description of the structure and prove some more interesting result related to the content of the last section.

Finally, in Section 6 we prove a result about the Ap\'ery set of a plane curve and its blowup, which generalizes and reinterprets a result in \cite{apery}. To motivate this to the reader, we present a more detailed overview on the content of this result and on its historical background.

Let $\mathcal{O}=\K[[X,Y]]/(F_1\cdots F_d)$, with $F_i$ irreducible polynomials, be the ring of a plane algebroid curve.
Let $S=\operatorname{v}(\mathcal{O}) $ be its value semigroup. 
 The minimal nonzero element $\e$ of $S$ is called \emph{multiplicity} and it is an invariant related to the multiplicity of the ring $\mathcal{O}$. %In general one has $e(\mathcal{O})\leq e$.
A fundamental invariant involved in the study of the equivalence classes of algebroid curves is the sequence of multiplicities of the successive blowups of the ring $\mathcal{O}$. % i.e.the sequence ($e(\mathcal{O}),e(\mathcal{O'}),e(\mathcal{O}^{(2)})\ldots$). 
 Two algebroid plane curves are formally equivalent if they have the same value semigroup \cite{waldi}. It is well-known that two plane algebroid branches (i.e. plane curves in the case $d=1$) have the same value semigroup if and only if
they have the same multiplicity sequence \cite{zariski}. Hence the problem of classification of plane curves can be considered equivalently in a semigroup setting.

In \cite{aperyold}, Ap\'{e}ry showed that the numerical semigroups of a plane branch and of its blowup can be determined by studying their respective Ap\'{e}ry sets (see also \cite{planealgebroid}). %The Ap\'{e}ry set of a numerical semigroup $S$ with respect to the multiplicity $e$ consists of the elements $s\in S$ such that $s-e \not \in S$. 
Denote by $\mathcal{O}'$ the blowup of $\mathcal{O}$. Ap\'{e}ry precisely proved that if 
%\begin{teo}
%\label{ap1}
%If $e$ is the multiplicity of the numerical semigroup $S$ associated to an algebroid branch $\mathcal{O}$ and 
$A=\{\omega_1,\ldots,\omega_{e}\}$ is the Ap\'{e}ry set of $S=\operatorname{v}(\mathcal{O})$, then $A'=\{\omega_1,\omega_2-e,\omega_3-2e,,\ldots,\omega_{e}-(e-1)e\}$ is the Ap\'{e}ry set of $S':= \operatorname{v}(\mathcal{O}')$ with respect to the same element $e$.
%\end{teo}
This theorem does not hold in the case of arbitrary numerical semigroups not associated to plane branches.

For $d>1$, in \cite{a-u}, it is shown how to associate a multiplicity tree to a good semigroup. The multiplicity tree is a tree where the vertices are the multiplicities of the value semigroups of iterated blowups of $\mathcal{O}$, and edges represent consecutive blowups.
%The multiplicity vector $\bs{e}=(e_1, \ldots, e_d)$ of a local good semigroup is the minimal nonzero element of the semigroup (minimal with respect the standard partial ordering of $\N^d$). 
After a finite number of blowups of $ \mathcal{O} $ one gets a semilocal non-local ring that is expressed as direct product of local rings. Until all the blowup rings (and equivalently their value semigroups) are local, the multiplicity tree is a path containing their multiplicity vector. After reaching a non-local ring, the tree branches out and each branch contains all the multiplicity vectors of all the semigroups that are components of the direct product and of their blowups.

%If we consider the sequence of blowups $\mathcal{O},\mathcal{O'},\mathcal{O}^{(2)}\ldots$, after a certain number of steps, we find a non-local blowup $\mathcal{O}^{(m)}$; if we denote by $S'=\operatorname{v}(\mathcal{O}^{(m)})$, $S'$ can be written as a direct product of good semigroups with a number of branches smaller then $d$. 
%Each level of the tree contains the vector of the multiplicities of the semigroup if the blowup is local; in the non-local case the tree branches out and each branch contains each of the multiplicity vector of the direct product. 
In the same paper \cite{a-u}, using the concept of Arf closure and Arf semigroup, the authors observed that any two algebroid curves are formally equivalent if they have the same multiplicity tree.

In \cite{apery}, %the notion of the Ap\'{e}ry set has been studied for value semigroups of plane curves with two branches ($d=2$). 
%Even if this set is infinite, it can be canonically partitioned in $e_1+e_2$ sets called
%\emph{levels}. 
generalizing Ap\'{e}ry's Theorem, the authors proved that, if $d=2$, $\mathcal{O}$ and its blowup $\mathcal{O}'$ are both local rings, and $A_i$ and $A'_i$ denote respectively the $i$-th level of the Ap\'{e}ry set of $\operatorname{v}(\mathcal{O})$ and of $\operatorname{v}(\mathcal{O'})$, then $A_i=A'_i+(i-1)\bs{e}$.

As a consequence, they showed how to determine the multiplicity tree of a good semigroup of a plane algebroid curve with two branches using the levels of the Ap\'{e}ry sets of iterated blowups.
Moreover, they showed also how to determine the value semigroup starting from a multiplicity tree. % In fact, in the two-branches case, when in the chain of blowups one reaches a non-local blowup, the tree breaks into two branches consisting of the multiplicity sequences of two numerical semigroups.
%Starting from the ends of these branches, it is possible to descend along the tree until the point where the two branches join. 
In order to do this, 
%Here 
they use a result by Garc\'ia \cite{garcia} which allows to determine the local value semigroup of a plane curve in $\N^2$ knowing its numerical projections (see \cite[Proposition 4.2]{apery}). %Then one can use the generalization of Ap\'{e}ry's Theorem to conclude. Globally this provides a complete characterization of value semigroups of plane algebroid curves with two branches. 

%%%%%%%%%%%%%%%%%%%%%%%%%%%%%%%%%%%%%%%%%%%%%%%%%%%%%%%%

Our purpose here is to give a proof of Ap\'ery's Theorem for value semigroups of plane curves with two branches also in the case $S$ is local and $S'$ is not. The key ingredient of this proof is the description of Ap\'ery sets of non-local good semigroups that we provide in Section 4. Observe indeed, that in \cite{apery} the Ap\'ery set was defined only in the local case. Our method allows to get the same result on the multiplicity tree without using Garc\'ia's result. Hopefully, the method we use here can be extended also to the case of an arbitrary number of branches $d$, for which a property analogous to that proved by Garc\'ia is unknown.

% We describe now the content of each section of this paper.

%Symmetric and almost symmetric good semigroups are interesting classes since, when they are value semigroups of algebroid curves, they correspond to those algebroid curves having respectively Gorenstein and Almost Gorenstein ring. For a more detailed description of these topics in connection to ring theory we refer to \cite{kunz}, \cite{delgado},  \cite{B-Fr}, \cite{a-u}.

\section{Preliminaries}
In this section we fix some notations, recall some known results and demonstrate some preliminary results that will be used in the following sections.\\
We use the symbol $\le$ to denote the partial ordering in $\mathbb N^d$: setting $\al=(\alpha_1, \alpha_2,\ldots,\alpha_d), \be=(\beta_1, \beta_2,\ldots,\beta_d)$, then $\boldsymbol{\alpha}\le \boldsymbol{\beta}$ if $\alpha_i\le \beta_i$ for all $i\in \{1,\ldots,d\}$.\\
%Given $\boldsymbol{\alpha},\boldsymbol{\beta}\in \mathbb N^d$, the infimum of the set $\{\boldsymbol{\alpha},\boldsymbol{\beta}\}$ (with respect to $\le$) will be denoted by $\bs{\alpha}\wedge \bs{\beta}$
Trough this paper, if not differently specified, when referring to minimal or maximal elements of a subset of $\mathbb N^d$, we refer to minimal or maximal elements with respect to $\le$. 

The element $\de$ such that $\delta_i= \min(\alpha_i, \beta_i)$ for every $i=1, \ldots, d$ is called the the infimum of the set $\{\boldsymbol{\alpha},\boldsymbol{\beta}\}$ and will be denoted by $\bs{\alpha}\wedge \bs{\beta}$.

Let $S$ be a submonoid of $(\mathbb N^d,+)$. We say that $S$ is a \emph{good semigroup} if

\begin{itemize}
	\item[(G1)] For every $\boldsymbol{\alpha},\boldsymbol{\beta}\in S$, $\boldsymbol{\alpha}\wedge \boldsymbol{\beta}\in S$;
	\item[(G2)] Given two elements $\boldsymbol{\alpha},\boldsymbol{\beta}\in S$ such that $\al\neq \be$ and $\alpha_i=\beta_i$ for some $i\in\{1,\ldots,d\}$, then there exists $\bs{\epsilon}\in S$ such that $\epsilon_i>\alpha_i=\beta_i$ and $\epsilon_j\geq \min\{\alpha_j,\beta_j\}$ for each $j\neq i$ (and if $\alpha_j\neq \beta_j$ the equality holds).
	\item[(G3)] There exists an element $\boldsymbol{c}\in S$ such that $\boldsymbol{c}+\mathbb N^d\subseteq S$.
\end{itemize}

A good semigroup is said to be \emph{local} if $\boldsymbol{0}=(0,\ldots,0)$ is its
only element with a zero component. 

By property (G1) it is always possible to define the element $\boldsymbol{c}:=\min\{\boldsymbol{\alpha}\in \mathbb Z^d\mid \boldsymbol{\alpha}+\mathbb N^d\subseteq S\}$; this element is called \emph{conductor} of $S$.
We set $\boldsymbol{\gamma}:=\boldsymbol{c}-\textbf{1}$.

A subset $E \subseteq \N^d$ is a \it relative ideal \rm of $S$ if $E+S \subseteq E$ and there exists $\al \in S$ such that $\al + E \subseteq S$. A relative ideal $E$ contained in $S$ is simply called an ideal. An ideal $E$
satisfying properties (G1), (G2) and (G3) is called a \it good ideal \rm. The minimal element $\boldsymbol{c}_E$ such that $\boldsymbol{c}_E + \N^d \subseteq E$ is called the \it conductor \rm of $E$. As for $S$, we set $\boldsymbol{\gamma}_E:=\boldsymbol{c}_E-\textbf{1}$.

We denote by $\bs{e}=(e_1,e_2,\ldots,e_d)$, the minimal element of $S$ such that $e_i>0$ for all $i\in I$.
The set $\e + S$ is a good ideal of $S$ and its conductor is $\boldsymbol{c} + \e$. Similarly for every $\om \in S$, the principal good ideal $E= \om + S$ has conductor $\boldsymbol{c}_E = \boldsymbol{c} + \om$.

\medskip

We will use trough all paper the following notation holding for any arbitrary subset $S \subseteq \N^d$.
We denote by $I$ the set of indexes $\{1,\ldots, d\}$. % where $d$ is the number of branches of the considered semigroup. 
Given $F\subseteq I$, $\al\in \N^d$, we set:
\begin{eqnarray*}
    \Delta^S_F(\al)&=&\{\be\in S \hspace{0.1cm}|\hspace{0.1cm} \beta_i=\alpha_i \text{ for } i\in F \text{ and } \beta_j>\alpha_j \text{ for } j\notin F\}.\\
    \widetilde{\Delta}^S_F(\al)&=&\{\be\in S \hspace{0.1cm}|\hspace{0.1cm} \beta_i=\alpha_i \text{ for } i\in F \text{ and } \beta_j\geq\alpha_j \text{ for } j\notin F\}\setminus\{\al\}.\\
    \Delta^S_i(\al)&=&\{\be\in S\hspace{0.1cm}|\hspace{0.1cm} \beta_i=\alpha_i \mbox{ and } \beta_j>\alpha_j \mbox{ for } j\neq i\}.\\
    \Delta^S(\al)&=&\bigcup_{i=1}^d\Delta_i^S(\al).
\end{eqnarray*}
In particular, for $S=\N^d$, we set $\Delta_F(\al):=\Delta_F^{\N^d}(\al)$.
   %\[ \Delta_F(\al):=\Delta_F^{\N^d}(\al)=\{\be\in \N^d \hspace{0.1cm}|\hspace{0.1cm} \beta_i=\alpha_i \text{ for } i\in F \text{ and } \beta_j>\alpha_j \text{ for } j\notin F\}.\]
   In general, given $F\subseteq I$, we denote by $\wh{F}$ the set $I \setminus F$. %We call $\wh{F}$ the \emph{orthogonal} set of $F$.

%GENERALI PER I SEMIGRUPPI   
We recall here some general properties concerning good semigroups 
and complementary sets of a good ideals proved in Section 1 and Section 2 of the paper \cite{GMM}. We refer to that paper for all the necessary proofs.
These properties will be widely used throughout the article; for this reason, we suggest to read these sections of \cite{GMM}, to find there further details and see graphical representations related to these properties. %and to familiarize yourself with the concepts.
  
  \begin{comment}
\begin{lemma}
\label{osssem}
Let $F,G \subseteq I$ and $E \subseteq \N^d$. We list some observations following directly from the preceding definitions:
\begin{enumerate}
    \item \label{osssem1} $\be\in \widetilde{\Delta}^E_F(\al)$ if and only if there exists $G\supseteq F$, such that $\be\in \Delta^E_G(\al)$. 
    \item \label{osssem2} $G\supseteq F$ if and only if, for every $\al \in \mathbb{N}^d$, $\tdE{G}{\al}\subseteq \tdE{F}{\al}$.
    \item \label{osssem3} If $\te \in \Delta^E_{F}(\al)$, then $\tdE{G}{\te} \subseteq \dE{G}{\al}$ for every $G \supseteq F$.
    \item \label{osssem4} Assume that $E$ is a good ideal. As a consequence of property (G2), if $\al\in E$ and $\dE{F}{\al}\neq \emptyset$, then $\tdE{\wh{F}}{\al}\neq \emptyset$.
    Equivalently, there exists $G\supseteq \wh{F}, G\neq I$, such that $\dE{G}{\al}\neq \emptyset$.
    \item \label{osssem5}
    Let $\al \in E$, $F \subsetneq I$ a set of indexes and assume there exists a set $H$ of cardinality $d-1$ containing $F$ and such that $\dE{G}{\al} = \emptyset$ for every $F \subseteq G \subseteq H$. Then $\dE{\widehat{F}}{\al} = \emptyset$.
 
\end{enumerate}
\end{lemma}
   \end{comment}
\medskip

 For any subset $A \subseteq S$, we say that two elements $\al, \be \in A$ are \it consecutive \rm in $A$ if whenever $\al \leq \de \leq \be$ for some $\de \in A$, then $\de = \al$ or $\de = \be$.

\begin{lemma} \label{minimidelta} Let $E \subseteq S$ be a proper good ideal. Then: %We list some observations following from results in \cite{GMM}.
\begin{enumerate}
    \item \label{osssem0} Let $\al \in S \setminus E$.
    Assume $\dE{F}{\al} \neq \emptyset$. As consequence of property (G1) of the good ideal $E$, $\tdE{\widehat{F}}{\al} = \emptyset. $
   \item \label{minimidelta1}Let $\al \in E$ and $\be \in \dE{F}{\al}$ and $\te \in \dE{G}{\al}$. If $F \cup G \subsetneq I$, then $\be \wedge \te \in \dE{F \cup G}{\al}$, while if $F \cup G = I$, then $\be \wedge \te = \al$.
    \item \label{minimidelta2}Let $\al \in E$ and $\be \in \dE{F}{\al}$ be consecutive to $\al$ in $E$. Then $\dE{H}{\al}= \emptyset$ for every $H \supsetneq F$.
    
    %\item \label{minimidelta3}Assume $\dE{F}{\al} = \emptyset$ and $F=G_1 \cup G_2$. Then either $\dE{G_1}{\al} = \emptyset$ or $\dE{G_2}{\al} = \emptyset$.
    
    %\item \label{minimidelta4}Assume there exists $F \subsetneq I$ such that $\dE{H}{\al} = \emptyset$ for every $H \supsetneq F$. Then $\dE{H}{\al} = \emptyset$ for every $H \subsetneq \wh F$.
    %\item \label{minimidelta5}Assume there exists a maximal set $F \subsetneq I$ such that $\dE{F}{\al} \neq \emptyset$ ($F$ is not necessarily unique). %there exists $F \subsetneq I$ (not unique) such that 
    %If $\dE{H}{\al} \neq \emptyset$ then either $H \subseteq F$ or $H \supseteq \wh F$. %\item Assume $\dE{F}{\al} \neq \emptyset$ and $\dE{H}{\al} = \emptyset$ for every $H \supsetneq F$ and for every $H \supsetneq \wh F$. Given $G\subset I$, $G\neq \emptyset$, then $\dE{ G}{\al} \neq \emptyset$ if and only if $G=F$ or $G= \wh F$.
\item \label{minimidelta6}
Let $\al \in \N^d$. Assume there exists $\be \in \dE{F}{\al}$ and that $\dE{H}{\al}$ is non-empty for some $H \subsetneq F$. Then there exists $T \subsetneq F$ such that $T \supseteq (F \setminus H)$ and $\dE{T}{\al} \neq \emptyset.$

\end{enumerate}
\end{lemma}

Let us now recall the definition of complete minimum; these elements are crucially involved in the definition of the Ap\'ery set of a good semigroup $S\subseteq \N^d$ with $d>2$.

\begin{definition} \rm \label{completeinfimum}
Let $S\subseteq \N^d$ be a good semigroup and let $A\subseteq S$ be any subset. We say that $\al\in A$ is a \it complete infimum \rm in $A$ if there exist $\be^{(1)},\ldots,\be^{(r)}\in A$, with $r \geq 2$, satisfying the following properties:
\begin{enumerate}
\item $\be^{(j)} \in \ds{F_j}{\al}$ for some non-empty set $F_j \subsetneq I$.
\item For every $j\ne k \in \lbrace 1, \ldots, r \rbrace$,  $\al = \be^{(j)} \wedge \be^{(k)} $. %(We note that this is equivalent to say $F_j\cup F_k=I$).
\item $\bigcap_{k=1}^r {F_k}= \emptyset$.
    \end{enumerate}
In this case we write $\al=\be^{(1)}\wt\be^{(2)}\cdots\wt\be^{(r)}$.
\end{definition}

In some proofs we will need to write an element as complete infimum of elements in specific directions. Let us therefore recall the following proposition, which descends directly by property (G2).

\begin{prop}
\label{propG2}
Let $S \subseteq \N^d$ be a good semigroup, $E \subseteq S$ a good ideal and $\al \in E$.  Suppose that there exists $\be \in \dE{F}{\al}$ for some $F\subsetneq I$. Then, there exist $\be^{(1)},\cdots,\be^{(r)}$ with $1 \leq r \leq |F|$, such that  $$\al = \be \wt \be^{(1)}\wt\be^{(2)}\wt\cdots\wt\be^{(r)}.$$
	In particular $\be^{(i)} \in \dE{G_i}{\al}$, with $G_i \supseteq \wh{F}$ and
	$ G_1 \cap G_2 \cap \cdots \cap G_r = \wh{F}. $
	
	From the proof of \cite[Proposition 1.7]{GMM}, it follows also that we can choose each $G_i$ such that $\dE{H}{\al}= \emptyset$ for every $H \supsetneq G_i$.
\end{prop}

Using the definition of complete infimum we can define a canonical partition of the complementary set $A$ of a given good ideal $E$. \\ %in a finite number of subsets 
%(the Ap\'{e}ry set with respect to an element $\om \in S$ corresponds to the case $ S \setminus \om+S$). 
Let $E$ a good ideal of a good semigroup $S$ and $A=S\setminus E$.
 Given $\al=(\alpha_1,\alpha_2,\ldots, \alpha_d)$ and $\be= (\beta_1,\beta_2,\ldots,\beta_d)$ in $\mathbb{N}^d$, we say that  $\al \leq \leq \be$ if and only if either $\al = \be$ or $\alpha_i<\beta_i$ for every $i\in\{1,\ldots, d\}$. In the second case we say that $\be$ \it dominates \rm $\al$ and use the notation $\al \ll \be$.

%$(\alpha_1,\alpha_2,\ldots, \alpha_d)\ne(\beta_1,\beta_2, \ldots, \beta_d)$ and $(\alpha_1,\alpha_2,\ldots, \alpha_d)\ll(\beta_1,\beta_2, \ldots, \beta_d)$, where the last means $\alpha_i<\beta_i$ for any $i\in\{1,\ldots, d\}$.\\
%%We denote with $\bs{g^{(i)}}\in \N^d$ the elements such that $g^{(i)}_j=0$ for any $j\neq i$ and $g^{(i)}_i=1$.
%In order to simplify the notation, if $\al\in S$ can be wrhttps://www.overleaf.com/project/5dd7e3419099b90001064734itten as $\al=\bei{i}\wedge\bei{j}$, with $\bei{i},\bei{j}\in \N^d$ and there exist $\lambda_i,\lambda_j \in \N$ such that $\bei{i}=\al+\lambda_i\bs{g^{(i)}}$ and $\bei{j}=\al+\lambda_j\bs{g^{(j)}}$, we will write $\al=\bei{i}\wt\bei{j}$.\\

\begin{definition} \rm \label{deflivelli}
Define $A$ as above. Set:
$$ B^{(1)}:=\{\al \in A : \al \  \mbox{is maximal with respect to} \le\le\},$$
$$C^{(1)}:= \{ \al \in B^{(1)} : \al=\be^{(1)}\wt\be^{(2)}\cdots\wt\be^{(r)} \mbox{ for } 1<r\leq d\mbox{ and } \bei{k} \in B^{(1)}\},$$
$$ D^{(1)}:=B^{(1)}\setminus C^{(1)}.$$
For $i>1$ assume that $D^{(1)},\dots , D^{(i-1)}$ have been defined and set inductively:
$$ B^{(i)}:=\{\al \in A \setminus (\bigcup_{j < i} D^{(j)}) : \al \  \mbox{is maximal with respect to} \le\le\},$$
$$C^{(i)}:= \{ \al \in B^{(i)} : \al=\be^{(1)}\wt\be^{(2)}\cdots\wt\be^{(r)} \mbox{ for } 1<r\leq d\mbox{ and } \bei{k} \in B^{(i)}\},$$
$$D^{(i)}:=B^{(i)}\setminus C^{(i)}.$$ 
By construction $D^{(i)}\cap D^{(j)}=\emptyset$, for any $i\neq j$ and, 
since the set $S \setminus A$ has a conductor, there exists $N \in \mathbb N_+$ such that $A=\bigcup_{i=1}^N D^{(i)}$.
As in \cite{GMM} we enumerate the sets in this partition in increasing order
setting $A_i:= D^{(N+1-i)}$. Hence 
$A=\bigcup_{i=1}^N A_i.$ 
We call the sets $A_i$ the \it levels \rm of $A$.
\end{definition}

Given $\om \in S$, we can consider the good ideal $E=\om+S$. In this case its complement $A= S \setminus E=\Ap(S, \om)$ is the Ap\'ery set of $S$ with respect to $\om$. The number of levels of an Ap\'ery set is equal to the sum of the components of the element $\om$ \cite[Theorem 4.4]{GMM}.

We recall that, if $\al, \be \in A$, $\al \ll \be$ and $\al \in A_i$, then $\be \in A_j$ for some $j>i$. Moreover, the last set of the partition is $A_N=\Delta(\ga_E)=\Delta^S(\ga_E)$.
If $S$ is local then $A_1=\{\boldsymbol {0}\}$. \\
Several basic properties of the Apéry set and of this partition in levels are listed in the \cite[Lemma 2.3]{GMM}.

\begin{comment}

\begin{lemma}
\label{prelem}
The sets $A_i$ satisfy the following properties: 
\begin{enumerate}
    \item \label{prelem1} Given $\al\in A_i$ with $i < n$, there exists $\be\in A_{i+1}$ such that $\al\ll\be$ or,  %necessarily
  $$\al=\be^{(1)}\wt\be^{(2)}\cdots\wt\be^{(r)}$$ where $1<r\leq d$, $\bei{k}\in A$, and at least one of them belongs to $A_{i+1}$.
    \item \label{prelem2} Given $\al\in A_i$, with $i\neq N$, there exists $\be\in A_{i+1}$ such that $\be\geq \al$.
    \item \label{prelem3} For every $\al\in A_i$ and $\be\in A_j$, with $j\geq i$, $\be \not \ll \al$.
    \item \label{prelem4} Given $\al\in A_i$ and $\be\in A$ such that $\be \geq \al$, then $\be\in A_{i}\cup\cdots\cup A_N$.
    \item \label{prelem5}
    Let $\al= \be^{(1)}\wt\be^{(2)}\cdots\wt\be^{(r)} \in A$ and assume that for every $k$, $\bei{k} \in A_i$. Then $\al \in A_h$ for some $h < i$.
    %Given $\al\in A_i$, $\be\in\ds{F}{\al}\cap A_i$, then for any $\de\in A\cap \td{\wh{F}}{\al}$, we get $\de\in A_{i+1}\cup\ldots\cup A_N$.
    \item \label{prelem6} Assume $\al, \be \in A$ are consecutive in $S$ or in $A$. If $\al\ll \be$, then there exists $i$ such that $\al \in A_i$ and $\be \in A_{i+1}$. If $\be\in \ds{F}{\al}$ for some $F \subsetneq I$, then there exists $i$ such that either $\al\in A_i$ and $\be\in A_{i+1}$ or $\al,\be\in A_{i}$.
 \end{enumerate}
\end{lemma}
\end{comment}

Now we restate two key theorems which will be used in many of the subsequent proofs. These are very helpful to control the levels of different elements. 
%\textcolor{red}{Uniformare la notazione tra 2.5 e 2.7 FATTO}
\begin{theorem}
\label{bianchi}
Let $S$ be a good semigroup, $E \subseteq S$ a good ideal and $A= S \setminus E$.
 Let $\de\in S$, $\te\in \Delta_G^S(\de)\cap A_h$ and assume $\ds{\widehat{G}}{\de}\subseteq A$. Let $\be \in \ds{F}{\de}$ with $\be$ and $\de$ consecutive and $F \supseteq \widehat{G}$.
\begin{enumerate}
\item If $\td{F}{\de} \subseteq A$, then $\be\in A_i$ with $i\leq h$;
\item If $\de\in A$ and $\td{\wh G}{\de}\subseteq A$ then $\de \in A_i$ with $i<h$.
\end{enumerate}
\end{theorem}

%Let $S$ be a good semigroup, $E \subseteq S$ a good ideal and $A= S \setminus E$.Let $\al\in S$, $\be\in \Delta_F^S(\al)\cap A_i$ and assume $\ds{\widehat{F}}{\al}\subseteq A$. Let $\bs{\theta}\in \ds{G}{\al}$ with $\bs{\theta}$ and $\al$ consecutive and $G \supseteq \widehat{F}$.
%\item If $\td{G}{\al} \subseteq A$, then $\bs{\theta}\in A_h$ with $h\leq i$;
%\item If $\al\in A$ and $\td{\wh F}{\al}\subseteq A$ then $\al\in A_h$ with $h<i$.

\begin{theorem}
\label{neri} 
Let $S$ be a good semigroup, $E \subseteq S$ a good ideal and $A= S \setminus E$.
Let $\al\in A_i$ and let $\te \in \ds{G}{\al}$ be consecutive to $\al$. Assume that $\tdE{\wh G}{\al} \neq \emptyset$. Then $\te \in A_i$. 
\end{theorem}

Next results have not been proved previously, hence we include also a proof of them.

\begin{lemma}
\label{neroinmezzo}
Let $S$ be a good semigroup and let $A= \bigcup_{i=1}^N A_i$ be the complement of a good ideal $E \subseteq S$.
Let $\al, \be \in A_i$, $\de \in S$, and assume $\al < \de < \be$. Then $\de \in A_i$.
\end{lemma}

\begin{proof}
It suffices to show $\de \in A$. By way of contradiction suppose $\de \in E$. Since $\al$ and $\be$ are in the same level and they are comparable, they must share at least a coordinate. Hence say that $\be \in \Delta^S_H(\al)$ and $\de \in \Delta^S_F(\al)$ with $F \supseteq H$.
Since $\de \in E$, by Lemma \ref{minimidelta}.\ref{osssem0}, $\widetilde{\Delta}^S_{\wh F}(\al) \subseteq A$. By property (G2), applied to $\al$ and $\be$ following Proposition \ref{propG2}, we can write $\al = \be \wt \be^1 \wt \ldots \wt \be^r$ with $\be^j \in \Delta^S_{G_j}(\al)$ and $\bigcap_{j=1}^r G_j = \wh H \supseteq \wh F$. Moreover, we can assume each $\be_j$ to be consecutive to $\al$ and therefore by Theorem \ref{neri} applied to $\de$ and $\al$, we get $\be_j \in A_i$. This is a contradiction since also $\al, \be \in A_i$.
\end{proof}

\begin{lemma}
\label{minimideldelta}
Let $S \subseteq \mathbb{N}^d$ be a good semigroup and let $A= \bigcup_{i=1}^N A_i$ be the complement of a good ideal $E \subseteq S$. Let $\al \in \mathbb{N}^d$ and suppose $\Delta^S(\al) \subseteq A$ and it is non-empty. Then, the minimal elements of $\Delta^S(\al)$ are all in the same level.
\end{lemma}

\begin{proof}
By property (G1), every set of the form $\Delta^S_k(\al)$ has only one minimal element.
Thus we assume that at least two sets of the form $\Delta^S_k(\al)$ are non-empty otherwise there is nothing to prove.
By possible permuting the indexes, suppose $\Delta^S_1(\al), \Delta^S_2(\al) \neq \emptyset$ and call $\be, \te$ their minimal elements. Thus $\de:= \be \wedge \te \in \Delta^S_{1,2}(\al)$. Say that $\be \in \Delta^S_F(\de)$ and $\te \in \Delta^S_G(\de)$ with $1 \in \wh G \subseteq F$ and $2 \in \wh F \subseteq G$. Suppose that $\be \in A_i$ and $\te \in A_h$ with $h \leq i$.
To prove that $h=i$ we %$\be$ and $\te$ are in the same level we 
apply Theorem \ref{bianchi} to the pair $\te, \de$ to show that $h \geq i$ (observe that by definition $\be$ is a minimal element in $\Delta^S_F(\de)$). Hence we need to verify that the assumption of Theorem \ref{bianchi} are satisfied, and show that $\Delta^S_{\wh G}(\de) \cup \widetilde{\Delta}^S_{F}(\de) \subseteq A$. %(and analogously also that $\Delta^S_{\wh F}(\de) \cup \widetilde{\Delta}^S_{G}(\de) \subseteq A$). 
Pick $\he \in \Delta^S_{\wh G}(\de) \cup \widetilde{\Delta}^S_F(\de)$. Then $\eta_1 = \delta_1 = \alpha_1$ and $\eta_j \geq \delta_j > \alpha_j$ for $j \neq 1,2$.  
If $\eta_2 > \delta_2$, then $\he \in \Delta^S_1(\al) \subseteq A$. In particular, since $2 \not \in F$, it follows that $\Delta^S_{\wh G}(\de) \cup \Delta^S_{U}(\de) \subseteq \Delta^S_1(\al) \subseteq A$ for every $U \supseteq F$ such that $2 \not \in U$.
%Thus, in the case when $F = I \setminus \lbrace 2 \rbrace$ we are done. %also $\widetilde{\Delta}^S_{F}(\de) = \Delta^S_{F}(\de) \subseteq A$ since $2 \not \in \wh G \cup F$, it follows that $\Delta^S_{\wh G}(\de) \cup \Delta^S_{F}(\de) \subseteq \Delta^S_1(\al) \subseteq A$.

%Otherwise 
Suppose there exists $\he \in \Delta^S_{U}(\de)$ with $F \cup \{2\} \subseteq  U$. We show that this will contradict the minimality of $\be$ in $\Delta^S_1(\al)$. Applying property (G2) to $\he$ and $\de$ as in Proposition \ref{propG2}, we find elements in some sets $\Delta_{H_1}^S(\de), \ldots, \Delta_{H_r}^S(\de)$ such that $H_1 \cap \ldots \cap H_r = \wh U$. In particular, since $2 \not \in \wh U$, we can find  
some element in a set $\Delta_H^S(\de)$ such that $H \cup U = I$ and $2 \not \in H$. Hence $H \nsubseteq F$. It follows that $F \subsetneq (H \cup F) \subseteq I \setminus \lbrace 2 \rbrace$ and by Lemma \ref{minimideldelta}.\ref{minimidelta1},
 $ \Delta_{H \cup F}^S(\de) \neq \emptyset $. A minimal element $\om$ in $\Delta_{H \cup F}^S(\de)$ is now an element of $\Delta^S_1(\al)$. Pick $j \in H \setminus F$. %and not in $F$. 
 Observing that $\beta_j > \omega_j = \delta_j$ we contradict the minimality of $\be$.
%in order to express $\de$ as complete infimum we find some element in a set $\Delta_H^S(\de)$ such that $2 \not \in H, H \cup G = I, H \cup U = I$. Hence $1 \in H$ and $H \nsubseteq F$. It follows that $F \subsetneq (H \cup F) \subseteq I \setminus \lbrace 2 \rbrace$
%and $ \Delta_{H \cup F}^S(\de) \neq \emptyset $. A minimal element of $\Delta_{H \cup F}^S(\de)$ is now an element of $\Delta^S_1(\al)$ smaller than $\be$, a contradiction.
\end{proof}

While it follows easily by definition, that if $\al \in A_i$ then there exists $\te \in A_{i+1}$ such that $\te \geq \al$, it is not straightforward to see that there there always exists also $\be \in A_{i-1}$ such that $\be \leq \al$ (if $d=2$ this is proved in \cite[Proposition 4]{DGM}).

\begin{prop}
\label{livelloinferiore}
Let $S$ be a good semigroup and let $A= \bigcup_{i=1}^N A_i$ be the complement of a good ideal $E \subseteq S$. For $i > 1$, given $\al \in A_i$ there exists always some $\be \in A_{i-1}$, $\be < \al$.
\end{prop}

\begin{proof}
We can restrict to assume $\al$ to be a minimal element in $A_i$ with respect to $\leq$. Looking for a contradiction we suppose that for every $\be \in A_{i-1}$, $\al \wedge \be \neq \be $.
It is always possible to find a $\be \in A_{i-1}$ such that $\de = \al \wedge \be$ is maximal. To conclude we have to show the existence of $\te \in A_{i-1}$ such that $\te \wedge \al > \de$.
Say that $\al \in \Delta^S_F(\de)$ and $\be \in \Delta^S_G(\de)$ with $G \supseteq \widehat{F}$. %Thus, we need to find $\te \in A_{i-1}$ such that $\te > \be$ and $\theta_j > \beta_j$ for some $j \not \in F.$
We consider different possible cases: \\
\bf Case 1: \rm $ \Delta^E_{H}(\be) \neq \emptyset$ for some $H \nsubseteq F$. \\
Let $\om \in \Delta^E_{H}(\be)$ and apply property (G2) as in Proposition \ref{propG2} to $\be$ and $\om$ to find non-empty sets $\ds{V_1}{\be}, \ldots, \ds{V_r}{\be}$ such that $V_1 \cap \cdots \cap V_r = \wh H$. 
Since $\wh F \nsubseteq \wh H $, we can find $\te \in \ds{V}{\be}$ with $\wh H \subseteq V $, $\wh F \nsubseteq V $.
%elements $\om^1, \ldots, \om^c \in S$ with $1 \leq c < d$ such that $\om^j \in \ds{V_j}{\be} $ and $V_1 \cap V_2 \cap \cdots \cap V_c = \wh H$. 
%Since $\wh F \nsubseteq \wh H $, there exists $j$ such that $\wh F \nsubseteq V_j $. Set $V:=V_j$ and $\te:=\om^j$. 
By Lemma \ref{minimidelta}.\ref{osssem0}, $\td{\widehat{H}}{\be} \subseteq A$. Hence $\te \in A $ and we may assume without restrictions $\te$ and $\be$ to be consecutive. Theorem \ref{neri} implies $\te \in A_{i-1}$. By construction $\te \geq \be$ and there exists $k \not \in F$ such that $\theta_k > \beta_k$. Therefore $\te \wedge \al > \de$. \\
%\bf Case 1: \rm $ \widetilde{\Delta}^E_{\widehat{F}}(\be) \neq \emptyset$. \\
%Let $\om \in \dE{U}{\be}$ with $U \supseteq \widehat{F}$ and apply property (G2) to $\be$ and $\om$ to find elements $\om^1, \ldots, \om^c \in S$ with $1 \leq c < d$ such that $\om^j \in \ds{V_j}{\be} $ and $V_1 \cap V_2 \cap \cdots \cap V_c = \wh U$. Since $\wh U \subseteq F $, there exists $j$ such that $\wh F \nsubseteq V_j $. Set $V:=V_j$ and $\te:=\om^j$. By ..., $\td{\widehat{U}}{\be} \subseteq A$ and, since $V \supseteq \wh U$, we get $\te \in A $ and we may assume without restrictions $\te$ and $\be$ to be consecutive. Theorem \bf neri \rm implies $\te \in A_{i-1}$. By construction $\te \geq \be$ and there exists $k \not \in F$ such that $\theta_j > \beta_j$. Therefore $\te \wedge \al > \de$. \\
\bf Case 2: \rm $ \Delta^E_{H}(\be) = \emptyset$ for every $H \nsubseteq F$. \\   %$ \widetilde{\Delta}^S_{\widehat{F}}(\be) \subseteq A $. \\
Notice that this hypothesis implies $\widetilde{\Delta}^S_{\widehat{F}}(\be) \subseteq A $.
Pick $\he$ such that $\de \leq \he < \be $ and $\he, \be $ and are consecutive. Say that $\be \in \Delta^S_{U}(\he)$ with $U \supseteq \widehat{F}$.   %Since $\be$ and $\he$ are consecutive, $\widetilde{\Delta}^S_{U}(\de) = \{ \be \} \cup \widetilde{\Delta}^S_{U}(\be) \subseteq A$. %By \cite[Proposition 1.4]{GMM}, $\Delta^S_{\wh U}(\he) \subseteq A$.
 Apply again property (G2) using Proposition \ref{propG2}. to $\he$ and $\be$ to find an element $\te \in \ds{V}{\he}$ with $V \supseteq \wh U$ and $\wh F \nsubseteq V $. Following Proposition \ref{propG2},
 %By construction as in the proof of \cite[Proposition 1.7]{GMM}, 
 we can also assume $\Delta^S_H(\he) = \emptyset$ for every $H \supsetneq V$ and therefore that $\te$ and $\he$ are consecutive.
  Since $\wh V \nsubseteq F$ we get
 $\widetilde{\Delta}^S_{\widehat{V}}(\be) \subseteq A $. We claim that this implies also $\widetilde{\Delta}^S_{\widehat{V}}(\he) \subseteq A $. Indeed if $\boldsymbol{\tau} \in \widetilde{\Delta}^E_{\widehat{V}}(\he)$, we would have $\boldsymbol{\tau} \wedge \be \in \widetilde{\Delta}^E_{\widehat{V} \cup U}(\he) = \widetilde{\Delta}^E_{U}(\he)$. Since $\be$ and $\he$ are consecutive, $\boldsymbol{\tau} \wedge \be = \be$ and hence $\boldsymbol{\tau} \in \widetilde{\Delta}^E_{\widehat{V}}(\be)$, a contradiction.
 
 Now, if $\he \in A_{i-1}$ we can just replace $\be$ by $\he$ and iterate, possibly using again Case 1 (notice that also $\al \wedge \he = \de$). Hence, suppose $\he \in A_{i-2} \cup E$. 
If $\he \in A_{i-2}$, since $\he, \be$ are consecutive, by Theorem \ref{neri} we must have $\widetilde{\Delta}^S_{\widehat{U}}(\he) \subseteq A$. If instead $\he \in E$, using that $\widetilde{\Delta}^S_{\widehat{V}}(\he) \cup \widetilde{\Delta}^S_{U}(\he) \subseteq A $, by \cite[Proposition 1.4]{GMM}, %\bf qui ora è lemma 2.1.6 ma direi di lasciare la ref dell'altro articolo, penso che usiamo questo fatto solo qui \rm, 
we get $\Delta^S_{\wh U}(\he) \cup \Delta^S_{V}(\he) \subseteq A$. Since $\Delta^S_H(\he) = \emptyset$ for every $H \supsetneq V$, it follows 
 $\widetilde{\Delta}^S_{V}(\he) \subseteq A$.
In both cases %we can say that also $\te$ and $\he$ are consecutive 
 $\te \in A$ and the assumptions of Theorem \ref{bianchi} are satisfied both if we apply it to $\te$, $\he$ or to $\be$, $\he$. This implies that
 $\te $ and $\be$ are in the same level. Finally, observe that $\theta_j \geq \eta_j \geq \delta_j$ and by the choice of $V$ there exists $k \not \in F$ such that $\theta_k > \beta_k = \delta_k$. Thus, we can conclude as in Case 1.
%Finally, assume $\he \in E$. By \cite[Proposition 1.4]{GMM}, we get $\Delta^S_{\wh U}(\he) \cup \Delta^S_{V}(\he) \subseteq A$. Hence, $\widetilde{\Delta}^S_{V}(\he) \subseteq A$ and we can conclude in the same way using Theorem \bf bianchi \rm.
%Again by \cite[Proposition 1.4]{GMM}, we get also $\Delta^S_{\wh V}(\he) \subseteq A$. %$\Delta^S_{H}(\he) = \emptyset$ for every $H \supsetneq U$. 
%By assumption $\widetilde{\Delta}^S_{U}(\be) \subseteq A$.
%Now, if $j \in \wh F$, $\theta_i \geq \eta_i \geq \delta_i = \beta_i$, but again, by the choice of $V$ there exists $k \not \in F$ such that $\theta_j > \beta_j$.
%If $\widetilde{\Delta}^S_F(\de) \subseteq A$, we can reduce to assume that $\al$ and $\de$ are consecutive and $\de \in E$ (consecutive elements in $A$ are either in the same or in consecutive levels). But in this case \bf teorema bianchi \rm implies that the level of $\al$ is less or equal than the level of $\be$, a contradiction. 
\end{proof}

\section{Duality of the Ap\'ery sets of symmetric and almost symmetric good semigroups}

In this section we extend to good semigroups in $\N^d$ the results on the duality of levels of Ap\'ery set of symmetric and almost symmetric good semigroups, proved in the case $d=2$ in \cite[Section 5]{DGM} and \cite[Section 5]{DGM2}. 
%\bf qui pure qualcosa può andare nell'introduzione? \rm

Let $S$ be a good semigroup.
Recall that an element $\al \in S$ is \it absolute \rm (sometimes also called maximal) if $\Delta^S(\al) = \emptyset$. Then $S$ is symmetric if and only if for every $\al \in \mathbb{Z}^d$, $\al \in S$ if and only if $\Delta^S(\g - \al)= \emptyset$. In a symmetric good semigroup, the absolute elements are dual with respect to $\g$ in the sense that if $\al$ in absolute then also $\g - \al \in S$ and it is absolute.

The set of pseudo-frobenius element of $S$ is the set $\pf(S)$ of elements $\al \in \mathbb{N}^d \setminus S$ such that $\al + \be \in S$ for every nonzero element $\be \in S$.
A good semigroup $S$ is almost symmetric if and only if $$\pf(S) = \Delta(\g) \cup \lbrace \al \in (\N^2 \setminus S) \ |\ \Delta^S(\g-\al) = \emptyset  \rbrace.$$
For an overwiew on properties of symmetric and almost symmetric good semigroup in connection with the Ap\'ery set we refer to \cite[Section 4]{DGM} and \cite[Section 4]{DGM2}.

In the case $d=1$, it is well-known that symmetric and almost symmetric numerical semigroups are characterized by duality properties on the elements of Ap\'ery set, or on the pseudo-Frobenius elements, with respect to the largest element in the set, see \cite[Proposition 4.10]{libropedro} and \cite[Theorem 2.4]{nari}.
%In the case $d=1$, given a numerical semigroup $S$ with smallest non zero element $e$ and Frobenius number $f(S)= \max (\mathbb{N} \setminus S)$, it is well-known that, writing $\Ap(S)=\{\omega_1, \dots, \omega_e\}$ with $\omega_i < \omega_{i+1}$, then $S$ is symmetric if and only if $w_i+w_{e-i+1}=w_e$. The set $\Ap(S)$ can be furtherly partitioned as $$ \Ap(S)=\{0=a_1, \dots, a_m\} \cup \{b_1, \dots, b_{t(S)-1}\} $$ where $a_m - e = f(S)$ and the set of pseudo-frobenius numbers of $S$ is $$ \pf(S)= \{b_1 - e, \dots, b_{t(S)-1} - e\} \cup \{f(S)\} $$ %(an element $f \in \mathbb{N} \setminus S$ is pseudo-frobenius if $f+s \in S$ for every nonzero $s\in S$). 
%In \cite[Theorem 2.4]{nari} it is proved that $S$ is almost symmetric if and only if $a_i+a_{m-i+1}=a_m$ and $b_j+b_{t(S)-j}=a_m + e$ and also if and only if $f_i+f_{t(s)-i}=f(S)$, where the $f_i$ are the pseudo-frobenius numbers of $S$ listed in increasing order.

% The analogous results are already well-known in the case $d=1$ which serves as motivation to all this section. Indeed, for a numerical semigroup $S$ with smallest non zero element $e$, it is well-known that, writing $\Ap(S)=\{\omega_1, \dots, \omega_e\}$ with $\omega_i < \omega_{i+1}$, then $S$ is symmetric if and only if $w_i+w_{e-i+1}=w_e$.

 In the case of symmetric and almost symmetric good semigroups some correspondent, but less intuitive, duality relations do exist for the levels of the partition of $\Ap(S)$. 
 In general if $A= \bigcup_{i=1}^N A_i$ is the complement of a good ideal $E$, for $\om \in \N^2$, define $\om^{\prime}:= \g_E - \om.$ For each level $A_i$,  set 
 $$ A_i^{\prime} := \left( \bigcup_{\om \in A_i} \Delta^S(\om^{\prime}) \right) \setminus \left( \bigcup_{\om \in A_j \mbox{, } j < i} \Delta^S(\om^{\prime}) \right).$$
 
 We want to prove the following theorem for arbitrary $d \geq 2$. For a good semigroup $S$ having Ap\'{e}ry set $ \Ap(S)=\bigcup_{i=1}^e A_i $, define $$ Z:= \pf(S) \cup \lbrace \boldsymbol{0} \rbrace, \quad W := \lbrace \boldsymbol{0} \rbrace \cup \Delta(\g + \e) \cup \lbrace \al \in \bigcup_{i=2}^{e-1} A_i \ |\ \al - \e \not \in \pf(S)  \rbrace.$$ It can be shown that the sets $Z$ and $W$ are complement of good ideals of some opportune semigroup, exactly as in \cite[Proposition 5.3]{DGM2}. Hence we can write their partitions $ Z=\bigcup_{h=1}^n Z_h $ and $ W=\bigcup_{i=1}^m W_i $. Then:
 
 \begin{theorem} 
\label{dualityalmost}
Let $S \subseteq \N^d$ be a good semigroup and let $A= \Ap(S)$. %Define the sets $Z$ and $W$ as above and let $ Z=\bigcup_{h=1}^n Z_h $ and $ W=\bigcup_{i=1}^m W_i $ be their partitions obtained as in Definition \ref{livelligenerici}. Set $$ Z_h^{\prime} = \left( \bigcup_{\de \in Z_h} \Delta^{S-M}(\g -\de) \right) \setminus \left( \bigcup_{\de \in Z_j \mbox{, } j < h} \Delta^{S-M}(\g-\de) \right)$$ and $$ W_i^{\prime} = \left( \bigcup_{\om \in W_i} \Delta^S(\g+\e-\om) \right) \setminus \left( \bigcup_{\om \in W_j \mbox{, } j < i} \Delta^S(\g+\e-\om) \right).$$ The following assertions are equivalent:
Then:
\begin{itemize}
    \item $S$ is symmetric if and only if $A_i^{\prime}= A_{e-i+1}$ for every $i=1, \ldots, e.$
    \item $S$ is almost symmetric if and only if $Z_h^{\prime}= Z_{n-h+1}$ for every $h=1, \ldots, n$ and $W_i^{\prime}= W_{m-i+1}$ for every $i=1, \ldots, m.$
\end{itemize}  
%\begin{enumerate}
%\item $S$ is almost symmetric.
%\item  $Z_h^{\prime}= Z_{n-h+1}$ for every $h=1, \ldots, n$ and $W_i^{\prime}= W_{m-i+1}$ for every $i=1, \ldots, m.$
%\item  $W_i^{\prime}= W_{m-i+1}$ for every $i=1, \ldots, m.$
%\end{enumerate}
\end{theorem}

\begin{proof}
This result can be proved exactly with the same method used in \cite[Theorem 9]{DGM} and \cite[Theorem 5.6]{DGM2} after proving the next general result that we state as Theorem \ref{dualmain}.
\end{proof}

\begin{ex}
Let us consider the good semigroup $S\subseteq \N^3$, having elements $\ll \boldsymbol{c} + \boldsymbol{1}$ equal to %\bf definire small elements \rm:
\small{\begin{align*}
\text{Small}(S)=&\{ ( 2, 2, 3 ), ( 2, 2, 4 ), ( 2, 2, 5 ), ( 2, 2, 6 ), ( 2, 4, 3 ), ( 2, 4, 4 ), ( 2, 4, 5 ), ( 2, 4, 6 ), ( 2, 5, 5 ), ( 2, 5, 6 ),\\
& ( 2, 6, 3 ), ( 2, 6, 4 ), ( 2, 6, 5 ), ( 2, 6, 6 ), ( 3, 2, 3 ), ( 3, 2, 4 ), ( 3, 4, 3 ), ( 3, 4, 4 ), ( 3, 4, 5 ), ( 3, 4, 6 ), \\
& ( 3, 5, 3 ), ( 3, 5, 4 ),( 3, 5, 5 ), ( 3, 5, 6 ), ( 3, 6, 5 ), ( 4, 2, 3 ), ( 4, 2, 4 ), ( 4, 2, 5 ), ( 4, 2, 6 ), ( 4, 4, 3 ),\\
& ( 4, 4, 4 ), ( 4, 4, 5 ), ( 4, 5, 4 ), ( 4, 5, 5 ), ( 4, 6, 3 ), ( 4, 6, 4 ), ( 4, 6, 6 ) \}
\end{align*}
}\normalsize
Using the procedure described in \cite[Proposition 1.6]{NG2}, it is possible to see that the \emph{length} and \emph{genus} of the good semigroup are both equal to the half of the sum of the components of the conductor. This imply that $S$ is a symmetric good semigroup (see \cite[Theorem 2.3]{delgado}).
In Figure \ref{disegnogrande} is represented the Ap\'{e}ry set of $S$. 
The levels that correspond to each other following the duality relation of Theorem \ref{dualityalmost} are represented with the same color.

\begin{figure}[H]
\centering
\includegraphics[scale=0.75]{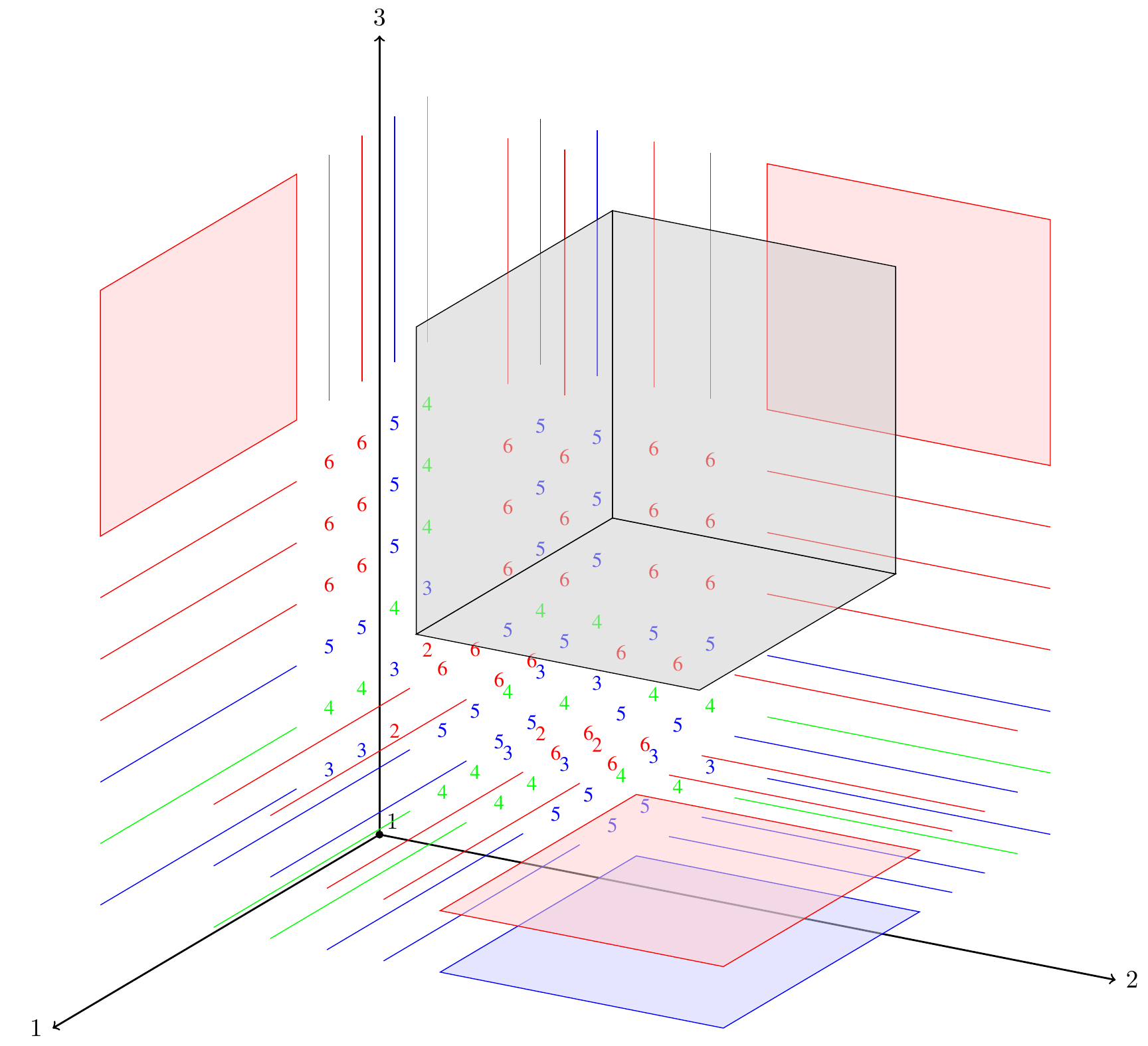}
\caption{\footnotesize The representation of the Ap\'{e}ry set of the good semigroup $S$. Levels 1-7 are black; levels 2-6 are red; levels 3-5 are blue; level 4 is green.}
\label{disegnogrande}
\end{figure}
\end{ex}

We give a general definition.

%$ \Ap(S)=\bigcup_{i=1}^e A_i $, where $e=e_1+e_2$. For $\om \in \N^2$, define $\om^{\prime}:= \g + \e - \om.$ Denoting $$ A_i^{\prime} := \left( \bigcup_{\om \in A_i} \Delta^S(\om^{\prime}) \right) \setminus \left( \bigcup_{\om \in A_j \mbox{, } j < i} \Delta^S(\om^{\prime}) \right),$$ in \cite[Theorem 5.3]{DGM} %9
% it is proved that $S$ is symmetric if and only if
%$A_i^{\prime}= A_{e-i+1}$ for every $i=1, \ldots, e,$ and hence, following this idea, the level $A_{e-i+1}$ is dual of the level $A_i.$

%Our aim is to apply Theorem \ref{dualitygen} to almost symmetric good semigroups, defining opportune subsets which are complementary of good relative ideals. For a good semigroup $S$ having Ap\'{e}ry Set $ \Ap(S)=\bigcup_{i=1}^e A_i $, we define: $$ Z:= \pf(S) \cup \lbrace \boldsymbol{0} \rbrace, $$ and: $$ W := \lbrace \boldsymbol{0} \rbrace \cup \Delta(\g + \e) \cup \lbrace \al \in \bigcup_{i=2}^{e-1} A_i \ |\ \al - \e \not \in \pf(S)  \rbrace.  $$ Recall that when $S$ is almost symmetric, $S-M$ is a good relative ideal of $S$. Also notice that when $S$ is symmetric, $Z=\lbrace \boldsymbol{0} \rbrace \cup \Delta(\g ) $ and $W = \Ap(S).$

\begin{definition}
 \label{symmetricduals}
  \rm Let $S$ be a good semigroup and let $A= \bigcup_{i=1}^N A_i$ be the complement of a good ideal $E \subseteq S$. %Let $\g_E$ be the conductor of $E$.
  We say that $A$ is a \it symmetric complement \rm of $E$ if: 
  \begin{itemize}
  \item $A_1 = \{ \boldsymbol{0} \}$.
  \item $\al \in E$ if and only if $\Delta^S(\g_E - \al) = \emptyset$.
  \item $\al \in A$ if and only if $\Delta^S(\g_E - \al) \subseteq A$ and it is non-empty.
  \end{itemize}
 \end{definition}

As consequence of \cite[Section 4]{DGM} and \cite[Section 4]{DGM2} it follows that the Ap\'ery set of a symmetric good semigroups and the sets $Z$ and $W$ in the case of an almost symmetric good semigroups are symmetric complements. This is proved for $d=2$ but the proof does not depend on the number $d$, hence it follows by the same argument for any $d$. Hence to prove Theorem \ref{dualityalmost} it is sufficient to prove next theorem.

\begin{theorem}
\label{dualmain}
Let $S$ be a good semigroup and let $A= \bigcup_{i=1}^N A_i$ be the complement of a good ideal $E \subseteq S$. Suppose $A$ is a symmetric complement. Then $A_i^{\prime}= A_{N-i+1}$ for every $i=1, \ldots, N.$
\end{theorem}
 
 \begin{proof}
 This proof can be done with the same exact method used for $d=2$ in \cite[Theorem 9]{DGM}, after proving the results in Lemma \ref{dualitylemma1} and Proposition \ref{dualityprop2}. Such results generalizes \cite[Lemma 4 and 5]{DGM} to the case of arbitrary $d \geq 2$.
 \end{proof}
 
 Hence we dedicate the remaining part of this section to prove Lemma \ref{dualitylemma1} and Proposition \ref{dualityprop2}. First we need some preliminary technical result.

 \begin{lemma}
 \label{tuttedirezioni}
 Let $S \subseteq \mathbb{N}^d$ be a good semigroup and let $A= \bigcup_{i=1}^N A_i$ be the complement of a good ideal $E \subsetneq S$. Let $\al \in A_i$ for $i < N$. For every $k= 1, \ldots, d$ there exists $\be^{(k)} \in A_{i+1}$ such that $\be^{(k)} \geq \al$ and $\beta^{(k)}_k > \alpha_k$.
 \end{lemma}
 
 \begin{proof}
 If there exists $\be \in A_{i+1}$ such that $\be \gg \al$ we are done. Hence, suppose this is not the case. %and that for every $\be \in A_{i+1}$ such that $\be \geq \al$, then $\beta_1 = \alpha_1$. 
 %Since $\al$ is not dominated by any element of $A_{i+1}$, it 
 It follows that $\al$ must be a complete infimum of elements in $A$ that are either in $A_i$ or in $A_{i+1}$, but not all in $A_i$. By definition of complete infimum we can find some element $\te \in A_i \cup A_{i+1}$ such that $\te \geq \al$ and $\theta_k > \alpha_k$. If $\te \in A_{i+1}$ we are done. Assume then $\te \in A_i$. Now, if $\te$ was dominated by some element of $A_{i+1}$, the same element would dominate also $\al$. Thus also $\te$ is a complete infimum of elements in $A_i \cup A_{i+1}$. Since at least one of such elements is in $A_{i+1}$, we conclude by choosing that element.
 \end{proof}

%By ...., if $S$ is symmetric, all the Apery sets of $S$ are symmetric complements. \bf menzionare l'altra applicazione agli almost simmetrici \rm.

 \begin{lemma}
 \label{spazivuotiduali} Let $S$ be a good semigroup and let $A= \bigcup_{i=1}^N A_i$ be the complement of a good ideal $E \subseteq S$. Suppose that $A$ is a symmetric complement, $\be \in A$ and $\widetilde{\Delta}^E_{G}(\be) = \emptyset$. Then there exists $k \in G$ such that $\Delta^S_k(\g_E -\be) \neq \emptyset$.
 %\textcolor{red}{Vogliamo provare che se un elemento $\beta$ non ha neri lungo una direzione chiusa $G$ allora il duale di $\beta'=\gamma_E-\beta$ ha un elemento su un iperpiano $k\in G$}
 \end{lemma}
 
 \begin{proof}
 Set $\be':= \g_E - \be$.
  For every $\de \in \mathbb{Z}^d$ such that $\be' \in \widetilde{\Delta}_{G}(\de)$, then $\de'= \g_E - \de \in \widetilde{\Delta}_{G}(\be)$. Hence, by definition of symmetric complement,
 $\Delta^S(\de) \neq \emptyset$ for any of such $\de$. %\textcolor{red}{Se fosse vuoto allora, per definizione di complemento simmetrico, $delta'$ che si trova nella direzione $G$ rispetto a $\beta$ sarebbe nero. Contro l'ipotesi che $\widetilde{\Delta}^E_{G}(\be) = \emptyset$}. 
 Moreover, since $\be \in A$, then $\Delta^S(\be') \neq \emptyset$. Hence the set $U= \{ j \in I : \Delta_j^S(\be') \neq \emptyset \}$ is non-empty. To get the thesis, we have to prove that $U \cap G \neq \emptyset$. By way of contradiction, suppose $U \subseteq \wh G$. %\textcolor{red}{$U$ è costituito dagli indici rappresentanti gli iperpiani di $\beta'$ che contengono qualcosa. Osserviamo che la tesi consisteva nel provare che nelle direzioni di G ci fosse almeno un iperpiano non vuoto; in altri termini vogliamo provare che U e G hanno almeno un elemento in comune; cioè $U\cap G\neq \emptyset$. Supponiamo che $U\cap G= \emptyset \Longleftrightarrow U \subseteq \widehat{G}$}%To prove this lemma, we need to produce an index $k \not \in U$ such that $\Delta_k^S(\be') \neq \emptyset$. 
 
 For each $j$, consider the half-line %$H_j:= I \setminus \{ j\}$ and 
 $T_j:= \{ \al \in \mathbb{Z}^d : \be' \in \Delta^S_{ I \setminus \{ j\} }(\al)\}$. %\textcolor{red}{Questi sono gli spazi duali degli iperpiani che contengono $\beta'$ nel loro delta.}. 
 For $j \in U$ we notice that, since $G \subseteq I \setminus \{ j\}$, each element $\de \in T_j$ is such that $\Delta^S(\de) \neq \emptyset$. %\textcolor{red}{Infatti Se $j\in U$, per ipotesi assurda $j\in \widehat{G} \Longrightarrow G\subseteq I\setminus\{j\}=\{\wh j\}$. Poiché per ipotesi, per ogni $\delta$ tale che $\beta'\in \Delta^S_{G }(\de)$ si ha $\Delta^S(\de) \neq \emptyset$}. 
 By construction, the $j$-th coordinate of the elements of $T_j$ can be arbitrarily small (possibly negative), therefore there exists a maximal element $\de^{(j)} \in T_j$ such that $\Delta^S_j(\de^{(j)}) = \emptyset$. %\textcolor{red}{VD.disegno, se non fosse vuoto, tramite la proprietà (G2) produrrebbe qualcosa sullo spazio (retta nel disegno) $T_j$, quindi potremmo spostarci su tale spazio fino a quando non troviamo un punto siffatto}. 
 Set $\de:= \bigwedge_{j \in U} \de^{(j)}$. It can be easily seen that $\be' \in \Delta_{\wh U}(\de)$. Thus, $\Delta^S(\de) \neq \emptyset$, since $\wh U \supseteq G$. (see Figure \ref{Figure 2} for a graphical representation in case $d=3$)
 \begin{figure}[H]
 \centering
 \includegraphics[scale=0.9]{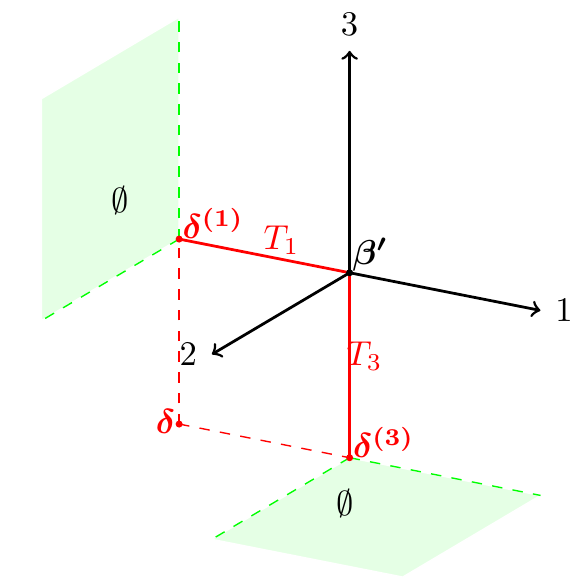}
 \caption{\footnotesize A graphical representation of costruction explained above. In this picture: $d=3$, $G=\{2\}, U=\{1,3\}$}
 \label{Figure 2}
 \end{figure}
 We divide now the proof in three parts, proving the following claims: \begin{enumerate}
     \item $\Delta^S_j(\de)= \emptyset$ for every $j \in U$. % and therefore %since $\Delta^S(\de) \neq \emptyset$, there must exist $k \not \in U$ such that $\Delta^S_k(\de) \neq \emptyset$.
     \item Given $\te\in \Delta_{\wh U}(\de)$ such that $\te \leq \be'$, then $\Delta^S_{ \{ k \} \cup H}(\te)= \emptyset$ for every $k \not \in U$, $H \subseteq U$ (including $H = \emptyset$, but excluding the case $\{ k \} \cup H = I$).
     \item Condition $1.$ and $2.$ together yield a contradiction.
 \end{enumerate}
  %The first thing to show is that $\Delta^S_j(\de)= \emptyset$ for every $j \in U$. 
  Let us prove $1.$ If $|U|=1$, then $\de = \de^{(j)}$ and the claim follows by definition of $\de^{(j)}$. Otherwise, assume by way of contradiction there exists $\om \in \Delta^S_j(\de)$. 
  Then $\omega_j = \delta_j = \delta^{(j)}_j < \beta'_j$ and $\omega_l > \delta_l = \delta^{(j)}_l = \beta'_l$ for $l \not \in U$. 
  %Now, using that $\Delta^S_k(\be') \neq \emptyset$ for every $k \in U$, by property (G1), we can find $\boldsymbol{\epsilon} \in \Delta^S_U(\be')$. For $k \in U$, $\beta'_k > \delta_k$. Thus also $\boldsymbol{\epsilon} \wedge \om \in \Delta^S_j(\de)$. Possibly replacing $\om$ by $\boldsymbol{\epsilon} \wedge \om$, we can assume that $\omega_k \leq  \delta^{(j)}_k = \beta'_k$ for every $k \in U$, $k \neq j$.
  Also we can suppose $\om$ to be minimal in $\Delta^S_j(\de)$ with respect to the coordinates in $\wh U$. 
  Since $\Delta^S_j(\de^{(j)})= \emptyset$, we must have $\omega_i \leq \delta^{(j)}_i = \beta'_i$ for some $i \in U$, $i \neq j$. 
  %Now, using that $\Delta^S_k(\be') \neq \emptyset$ for every $k \in U$, by property (G1), we can find $\boldsymbol{\epsilon} \in \Delta^S_U(\be')$. For $k \in U$, $\beta'_k > \delta_k$. Thus also $\boldsymbol{\epsilon} \wedge \om \in \Delta^S_j(\de)$. %The choice of the element $\om$ minimal with respect to the coordinates in $\wh U$ implies that $\boldsymbol{\epsilon} > \om $.
  We want to find an element $\te \in S$, such that $\te > \om$, $\theta_i > \omega_i$, and $\theta_j = \omega_j$. Iterating this process we find eventually an element in $\Delta^S_j(\de^{(j)})$ and this is impossible. The element $\te$ is constructed using property (G2) as follows. %Let $V$ be the set of indexes $i \in U$ such that $ \delta_i < \omega_i \leq \delta^{(j)}_i$.

  Notice that for each $k \in U \setminus \{ j \}$ 
  %such that $ \delta_i < \omega_i \leq \delta^{(j)}_i$, 
  we can find $\boldsymbol{\tau}^{(k)}$ in the line $T_k \cup \{ \be' \}\cup \Delta_{I \setminus \{ k\}}(\be') $ such that $\tau^{(k)}_k = \omega_k $ and $\boldsymbol{\tau}^{(k)} > \de^{(k)}$.
  Define $\boldsymbol{\tau}:= \bigwedge_{k \in U \setminus \{ j \}}\boldsymbol{\tau}^{(k)}$.
  Observe that $\boldsymbol{\tau} \leq \boldsymbol{\tau}^{(i)} \leq \be'$ and $\tau_k = \min \{ \beta'_k, \omega_k \}$.
  Hence, $\delta^{(k)}_k < \tau_k \leq \beta'_k  $ for every $k \in U \setminus \{ j\}$.
  %\bf problema: si può fare solo se $\omega_k \leq \de^{(j)}_k$. \rm
  By definition of $\de^{(k)}$ and by the assumption that $\Delta^S_k(\be') \neq \emptyset$ for every $k \in U$, we obtain
  %$\Delta^S_k(\boldsymbol{\tau}^{(k)}) \neq \emptyset$. %Let $V$ be the set of indexes $i \in U$ such that $ \delta_i < \omega_i \leq \delta^{(j)}_i$.
   %to be the infimum of all such $\boldsymbol{\tau}^{(i)}$.
  %Now, clearly 
  that for each $k \in U \setminus \{ j \}$, $\Delta_k^S(\boldsymbol{\tau}) \neq \emptyset$. Therefore by property (G1), $\Delta^S_{U \setminus \{ j \}}(\boldsymbol{\tau}) \neq \emptyset$.
  %Since $\boldsymbol{\tau}_k \geq \omega_k$ for each $k \in U$ and $\om$ is minimal with respect to the coordinates in $\wh U$, there exists some element in 
  Pick $\al \in \Delta^S_{U \setminus \{ j \} }(\boldsymbol{\tau})$. Since $\tau_k \geq \omega_k$ for each $k \in U$ and $\tau_l = \beta'_l = \delta_l $ for each $l \not \in U$, we get that $\al \wedge \om \in \Delta^S_j(\de)$.
  The choice of the element $\om$ minimal with respect to the coordinates in $\wh U$ implies that $\al > \om $. 
  Thus $ \al \in \Delta^S_{V }(\om) $ with $V \supseteq U \setminus \{ j \}$, and $j \not \in V$ because $\tau_j = \beta'_j > \delta^{(j)}_j = \omega_j$.
  Define $\te$ by applying property (G2) to $\om$ and $\al$ in such a way to find $\te \in \Delta^S_H(\om)$, with $j \in \wh V \subseteq H $ and $i \not \in H$.

 To prove $2.$ observe that such condition is true for $\be'$
 by an iterated application of %\cite[Proposition 2.6]{GMM} 
 Lemma \ref{minimidelta}.\ref{minimidelta6} choosing $S$ as good ideal of itself. Indeed, by the fact that $\Delta^S_{j}(\be') \neq \emptyset$ if and only if $j  \in U$, we get that $\Delta^S_{ \{ k,j \}}(\be')= \emptyset$ for every $k \not \in U$, $j \in U$. Iterating, we obtain also $\Delta^S_{ \{ k \} \cup H}(\be')= \emptyset$ for $k \not \in U$ and $H \subseteq U$. 
 Let now $\boldsymbol{\epsilon}_1, \ldots, \boldsymbol{\epsilon}_d$ be the elements of the canonical basis of $\mathbb{Z}^d$ as $\mathbb{Z}$-module. Starting from $\be'$, we proceed inductively assuming our claim true for $\te + \boldsymbol{\epsilon}_l$ with $l \in U$ and proving it for $\te$.
 %claim that the same condition is satisfied also by any element $\te \in \Delta_{\wh U}(\de)$ %\cap \{ \te \in \mathbb{Z}^d : \te \geq \be' \}$. 
 %such that $\te \leq \be'$. 
  %proceed inductively and, starting from $\be'$, we assume the claim true for $\te + \boldsymbol{\epsilon}_l$ with $l \in U$ and we prove it for $\te$. 
  We only need to show that $\Delta_k^S(\te) \neq \emptyset$ if and only if $k \in U$, and then apply again Lemma \ref{minimidelta}.\ref{minimidelta6} in the same way as done for $\be'$. In the case when $k \not \in U$, if there were some $\om \in \Delta^S_k(\te)$, we would have %$\omega_k = \theta_k = (\theta + \epsilon_l)_k$, $\omega_l \geq $
 $\om \in \Delta^S_{ \{k,l \} }(\te + \boldsymbol{\epsilon}_l) \cup \Delta^S_{ \{k \} }(\te + \boldsymbol{\epsilon}_l)$ which contradicts the inductive hypothesis.
 For $j \in U$, observe that
 by construction $\te= \bigwedge_{j \in U} \te^{(j)}$ for some $\te^{(j)}  \in T_j \cup \{ \be' \}$ such that $\te^{(j)}> \de^{(j)} $ and $\theta^{(j)}_j = \theta_j$. By definition of $\de^{(j)}$, $\Delta^S_j(\te^{(j)}) \neq \emptyset$, and therefore also $\Delta^S_j(\te) \neq \emptyset$. % for every $j \in U$. 
 %Once proved this claim, we show that $\Delta^S_j(\de)= \emptyset$ for every $j \in U$. If $|U|=1$, then $\de = \de^{(j)}$ and this follows by definition. Otherwise, assume there exists $\om \in \Delta^S_j(\de)$ (and assume it to be minimal with respect to the coordinates in $\wh U$). Then $\omega_j = \delta_j = \delta^{(j)}_j$ and $\omega_l > \delta_l = \delta^{(j)}_l$ for $l \not \in U$. Since $\Delta^S_j(\de^{(j)})= \emptyset$, we must have $\omega_i \leq \delta^{(j)}_i$ for some $i \in U$, $i \neq j$. We want to find an element $\om^* \in S$, such that $\om^* > \om$ and $\omega^*_i > \omega_i$. Iterating this process we find eventually an element in $\Delta^S_j(\de^{(j)})$ and this is impossible. The element $\om^*$ is constructed using property (G2) as follows. Noticing that $ \delta_i < \omega_i \leq \delta^{(j)}_i$, we can find $\boldsymbol{\tau} \in T_i \cap \{ \be' \}$ such that $\tau_i = \omega_i$ and $\boldsymbol{\tau} > \de^{(i)}$. Hence, by definition of $\de^{(i)}$, $\Delta^S_i(\boldsymbol{\tau}) \neq \emptyset$. Define then $\om^*$ by applying property (G2) to $\om$ and one element in $\Delta^S_i(\boldsymbol{\tau})$ \bf da controllare \rm
 
 Finally, we prove $3$. By $1.$, since $\Delta^S(\de) \neq \emptyset$, 
necessarily there must exist $k \not \in U$ such that $\Delta^S_k(\de) \neq \emptyset$. Pick $\al \in \Delta^S_k(\de)$. %First suppose $|\wh U|>1$. 
Let $\te \in \mathbb{Z}^d$ be defined such that $\theta_j=\beta'_j= \delta_j $ if $j \in \wh U$ and $\theta_j= \alpha_j \wedge \beta'_j$ if $j \in U$. By construction, $\te \in \Delta_{\wh U}(\de)$ and $\te \leq \be'$. Moreover, $\al \in \Delta^S_{\{ k \} \cup H}(\te)$ for some $H \subseteq U$. If $\{ k \} \cup H \neq I$ we find a contradiction by $2.$ Otherwise, we must have $H=U = I \setminus \{k \}$ and $\te = \al \in S$. But, as observed in the proof of $2.$, $\Delta^S_j(\te) \neq \emptyset$ for every $j \in U$, and by property (G1) (Lemma \ref{minimidelta}.\ref{minimidelta1}), $\Delta^S_U(\te) \neq \emptyset$. By property (G2), $\widetilde{\Delta}^S_{\wh U}(\te)= \widetilde{\Delta}^S_{k}(\te) \neq \emptyset$. Again this contradicts $2.$
 %Now, since $\Delta^S(\de) \neq \emptyset$, there must exist $k \not \in U$ such that $\Delta^S_k(\de) \neq \emptyset$. We show that this contradicts the property %we proved above for all the elements $\te \in \Delta_{\wh U}(\de)$ %\cap \{ \te \in \mathbb{Z}^d : \te \geq \be' \}$. 
 %such that $\te \leq \be'$. 
 \end{proof}
 
 We are finally ready to prove Lemma \ref{dualitylemma1} and Proposition \ref{dualityprop2}.
 
 \begin{lemma}
 \label{dualitylemma1}
  Let $S$ be a good semigroup and let $A= \bigcup_{i=1}^N A_i$ be the complement of a good ideal $E \subseteq S$. Suppose $A$ to be a symmetric complement. Let $\al \in A_i$, then
  $$ \Delta^S(\g_E - \al) \cap A_j = \emptyset, \mbox{ for every } j< N-i+1.    $$
 \end{lemma}
 
 \begin{proof}
 We work by decreasing induction on $i$ starting by $i=N$. In the basis case there is nothing to prove. Assume the thesis to be true for the elements in the level $A_{i+1}$. %For simplicity call $\al':= \g_E - \al$. 
 Pick $\te \in \Delta^S(\g_E - \al)$ and without loss of generality say that $\te \in \Delta^S_1(\g_E - \al)$. Since $A$ is a symmetric complement we can say that $\te \in A_h$ for some $h$. By Lemma \ref{tuttedirezioni}, we can find $\be \in A_{i+1}$ such that $\be > \al$ and $\beta_1 > \alpha_1$. To conclude, it suffices to prove that there exists some element in $\Delta^S(\g_E - \be) \cap A_t$ with $t < h$. Indeed, assuming by way of contradiction $h < N-i+1$, one would get $t < N-i$ contradicting the inductive hypothesis. Clearly $\g_E - \be < \g_E - \al $. %To fix directions say that $\g_E - \be \in \Delta^S_F(\g_e -\al)$ with $1 \not \in F$ and $F$ possibly empty (i.e. possibly $\be \gg \al$).
 Let $\de$ be a minimal element in $\Delta^S(\g_E -\be)$ and suppose $\de \in \Delta^S_k(\g_E - \be)$. It is easy to see that $\te \gg \g_E - \be$. Hence $\de \wedge \te \in \Delta^S_k(\g_E -\be)$ and, by minimality of $\de$, we get $\de \leq \te$.
 If $\de \ll \te$, then $\de \in A_t$ with $t < h$. Otherwise, 
 %set $\al':= \g_E - \al$, $\be':= \g_E - \be$ and observe that $\theta_1 = \alpha'_1 > \beta'_1 $, %$\theta_k \geq \alpha_k \geq \beta_k = \delta_k$ 
 %and $\theta_j > \alpha'_j \geq \beta'_j $ for $j \neq 1$. 
 %Hence $\de \wedge \te \in \Delta^S_k(\be)$ and, by minimality of $\de$, we get $\de \leq \te$.
 %Therefore 
 $\te \in \Delta^S_H(\de)$ with $k \not \in H$. %We have to show that if $\te$ and $\de$ are in the same level we get a contradiction. 
 Notice now that $$ \widetilde{\Delta}^S_{\wh H}(\de) \subseteq \Delta^S_k(\g_E - \be) \subseteq A.    $$
 By application of Theorem \ref{bianchi} to $\te$ and $\de$, we get that $\de$ is in a level strictly smaller than the level of $\te$.
 \end{proof} 
 
 \begin{prop}
 \label{dualityprop2}
 Let $S$ be a good semigroup and let $A= \bigcup_{i=1}^N A_i$ be the complement of a good ideal $E \subseteq S$. Suppose $A$ to be a symmetric complement. Let $\al \in A_i$, then the minimal elements of 
  $ \Delta^S(\g_E - \al)$ are in $A_{N-i+1}$.
 \end{prop}
 
 \begin{proof}
 We work by increasing induction on $i$. If $i=1$, since $A_1 = \{ \boldsymbol{0} \} $, we have $\Delta^S(\g_E - \boldsymbol{0}) = A_N$ and the thesis is true. Hence, we assume the thesis to be true for the elements in the level $A_{i-1}$ and we prove it for $A_i$. By Lemma \ref{dualitylemma1}, $\Delta^S(\g_E - \al) \cap A_j = \emptyset$ for every $j < N-i+1$.
 By Proposition \ref{livelloinferiore}, there exists $\be \in A_{i-1}$ such that $\be < \al$. We can also assume that there are not other elements of $A_{i-1}$ strictly between $\be$ and $\al$. Say that $\al \in \Delta^S_F(\be)$ with possibly $F = \emptyset$ (i.e. $\al \gg \be$). We claim then that $\widetilde{\Delta}^E_{\wh F}(\be) = \emptyset$.
 Indeed, if this were not the case, using property (G2) as in Proposition \ref{propG2}, we can express $\be$ as complete infimum of (consecutive) elements in some directions $\Delta^S_{H_1}(\be), \ldots, \Delta^S_{H_t}(\be)$ with $\bigcap_{l=1}^t H_l = \wh V$ for some $V \supseteq \wh F$. Hence, there exists $H = H_l$ such that $H \cup F \neq I$. 
 By Theorem \ref{neri}, we could find $\te \in A_{i-1} \cap \Delta^S_{H}(\be)$. % with $H \cup F \neq I$.
 Hence, $\be < \te \wedge \al \leq \te$, and by Lemma \ref{neroinmezzo} $\te \wedge \al \in A_{i-1}$, a contradiction with the assumption on $\be$.
 
 Now, recall that by Lemma \ref{minimideldelta} all the minimal elements of  $\Delta^S(\g_E - \al)$ are in the same level, thus it is enough to determine the level of only one of them. 
 
 Let $\om$ by a minimal element of $\Delta^S(\g_E - \al)$. By Lemma \ref{spazivuotiduali}, since $\widetilde{\Delta}^E_{\wh F}(\be) = \emptyset$, then there exists $k \in \wh F$ such that $\Delta^S_k(\g_E -\be) \neq \emptyset$.
 %$\bigcup_{k \not \in F}\Delta^S_k(\g_E -\be) \neq \emptyset$
 Therefore we can find a minimal element $ \de \in \Delta^S_k(\g_E - \be)$. % with $k \not \in F$. 
 By inductive hypothesis, $\de \in A_{N-i+2}$. We need to show that $\om$ is in a level strictly smaller than the level of $\de$. This will imply $\om \in A_{N-i+1}$ by Lemma \ref{dualitylemma1}. %We can assume without loss of generality that $\beta'_1 > \alpha'_1$. If $\de \in \Delta^S_k(\g_E - \be)$
 By assumption on $\de$ we have that $\de \gg \g_E - \al$. This implies the result by the same exact argument used at the end of the proof of Lemma \ref{dualitylemma1}.
 %If $\de \gg \g_E - \al$ we obtain this result by the same exact argument used in the proof of Lemma \ref{dualitylemma1}. Otherwise, this implies that $\g_E -\be \in \Delta^S_F(\g_E - \al)$ with $F \neq I$ and $\bigcup_{k \not \in F}\Delta^S_k(\g_E -\be) = \emptyset$. \bf continuare \rm
 \end{proof}

 \section{The Ap\'ery set of a non-local good semigroup in $\mathbb{N}^d$}
 \label{nonlocal}
 
 The partition in level of the complement of a good ideal described in \cite{GMM} and recalled here in Section 2 is perfectly well-defined also for non-local good semigroup. From \cite[Theorem 2.5]{a-u} every good semigroup can be expressed as a direct product of local good semigroups. The nice structure of non-local good semigroups as direct products allow us to give a more precise description of the levels of the partition in terms of the levels of partitions in the direct factors.
 We do this in Theorem \ref{thmnonlocallevels}.
 
 Our method consists of proving all the results for a direct product of two arbitrary good semigroups (not necessarily local). Everything proved in this setting can be extended by a finite number of iterations to any non-local good semigroup. 
 
 Our setting is the following:
Let $d_1, d_2 \geq 1$ and let $S_1 \subseteq \mathbb{N}^{d_1} $, $S_2 \subseteq \mathbb{N}^{d_2}$ be two good semigroups, not necessarily local. Consider the non-local good semigroup $S := S_1 \times S_2  \subseteq \mathbb{N}^d$ where $d= d_1+d_2$. Each element of $S$ is expressed in the form $ (\al^{(1)}, \al^{(2)})$ with $\al^{(1)} \in S_1$, $\al^{(2)} \in S_2$. We also define set of indexes $I_1:= \lbrace 1, \ldots, d_1 \rbrace $, $I_2:= \lbrace d_1+1, \ldots, d_1+d_2 \rbrace $.

Let $E_1$, $E_2$ be proper good ideals, respectively of $S_1$ and $S_2$. It is easy to check that $E := E_1 \times E_2 $ is a good ideal of $S$. If $A :=  S \setminus E$, $ A^{(1)} :=  S_1 \setminus E_1$, and $A^{(2)} :=  S_2 \setminus E_2$ then 
$$  A = (A^{(1)} \times S_2) \cup (S_1 \times A^{(2)}). $$

%Pick $\e= (\e^{(1)}, \e^{(2)}) \in S$ and assume $(\e^{(1)}, \e^{(2)}) \gg (\boldsymbol{0}, \boldsymbol{0})$.
%Let $A:=\Ap(S,\e )$ and call $\Ap(S_1):= \Ap(S_1, \e^{(1)})$, $\Ap(S_2):= \Ap(S_2, \e^{(2)})$
%We want to describe the level of $A$ in term of the levels of $\Ap(S_1)$ and $\Ap(S_2)$. It follows easily by the definitions that $$ \Ap(S)= (\Ap(S_1) \times S_2) \cup (S_1 \times \Ap(S_2)).  $$
%Setting $E_1:=S_1\setminus \Ap(S_1)$ and $E_2:=S_2\setminus \Ap(S_2)$, we get 
%$$  E := S \setminus A = E_1 \times E_2.   $$

%\bf CAMBIARE QUESTA NOTAZIONE \rm
 %it will be useful to define the set of indexes $I_1:= \lbrace 1, \ldots, d_1 \rbrace $, $I_2:= \lbrace d_1+1, \ldots, d_1+d_2 \rbrace $, and $I:= I_1 \cup I_2$. 
 We prove the following lemmas.

\begin{lemma}
\label{directprod1}
Pick $ (\al^{(1)}, \al^{(2)}) \in S$. Then $\al^{(2)} \in E_2$ if and only if there exists $\he^{(1)} > \al^{(1)}$ such that $(\he^{(1)}, \al^{(2)}) \in E$.
\end{lemma}

\begin{proof}
%If $\al^{(2)} \in E_2$, then $\al^{(2)} - \e^{(2)} \in S_2$ and thus $(\al^{(1)}, \al^{(2)}-\e^{(2)}) + (\e^{(1)},\e^{(2)}) = (\al^{(1)} + \e^{(1)}, \al^{(2)}) \in E$.
%If $\al^{(2)} \in E_2$, 
Since $E_1$ is a good ideal we can always find $\he^{(1)} \in E_1$ such that $\he^{(1)} > \al^{(1)}$.
Both implications follow now since $E= E_1 \times E_2$.
\end{proof}

\begin{lemma}
\label{directprod2}
Fix an element $\al^{(1)} \in A^{(1)}$. Let $\te^{(2)} \leq \al^{(2)}$ be two consecutive elements of $S_2$.
Then $ (\al^{(1)}, \te^{(2)}), (\al^{(1)},\al^{(2)}) \in A$ and they are consecutive in $S$. Suppose $(\al^{(1)}, \te^{(2)}) \in A_i$. Then $(\al^{(1)},\al^{(2)}) \in A_i$ if and only if there exists $\de^{(2)} \in E_2$ such that $\de^{(2)} \wedge \al^{(2)} = \te^{(2)}$ (in particular, also if $\de^{(2)} = \te^{(2)} \in E_2$). Otherwise $(\al^{(1)},\al^{(2)}) \in A_{i+1}$.
\end{lemma}

\begin{proof}
Since $\al^{(1)} \in A^{(1)}$, then $(\al^{(1)}, \te^{(2)}), (\al^{(1)},\al^{(2)}) \in A$ and they are clearly consecutive since so are $\al^{(2)}$ and $\te^{(2)}$ in $S_2$. Hence $(\al^{(1)},\al^{(2)}) \in A_i \cup A_{i+1}$. Looking at the coordinates we can find a set of indexes $F \supseteq I_1$ such that $(\al^{(1)},\al^{(2)}) \in \Delta^S_F(\al^{(1)},\te^{(2)})$.

Assume there exists $\de^{(2)} \in E_2$ such that $\de^{(2)} \wedge \al^{(2)} = \te^{(2)}$. Thus, by %\cite[Lemma 1.5]{GMM} 
Lemma \ref{minimideldelta}.1 there exists $G \supseteq I_1$ such that $(\al^{(1)}, \de^{(2)}) \in \Delta^S_G(\al^{(1)}, \te^{(2)}) \cup \{(\al^{(1)}, \te^{(2)})\}$ and $F \cup G = I$.
By Lemma \ref{directprod1}, there exists $\he^{(1)} > \al^{(1)}$ such that $(\he^{(1)}, \de^{(2)} ) \in E$. This is saying that $(\he^{(1)}, \de^{(2)}) \in \Delta^E_H(\al^{(1)}, \te^{(2)})$ with $H \nsupseteq I_1$ and $H \cap I_2 = G \cap I_2$. In particular, $F \supseteq \widehat{H}$. Since $(\he^{(1)}, \de^{(2)} ) \in E$, by %\cite[Remark 1.6]{GMM} \bf nuova ref per questo fatto \rm
Lemma \ref{minimidelta}.\ref{osssem0},
$\widetilde{\Delta}^S_{\widehat{H}}(\al^{(1)}, \te^{(2)}) \subseteq A$. By %\cite[Theorem 2.8]{GMM}(neri),
Theorem \ref{neri}
$(\al^{(1)}, \al^{(2)}) \in A_i$.

Conversely, suppose that no element $\de^{(2)} \in E_2$ satisfying such property exists. Let $G$ be a set of indexes such that $F \cup G = I$. 
%( $G \supseteq \widehat{F}$)
If there were some element $(\he^{(1)},\he^{(2)}) \in \Delta^E_G(\al^{(1)}, \te^{(2)})$, then $\he^{(2)} \in \Delta^{E_2}_{G \cap I_2}(\te^{(2)})$ and this would be a contradiction (notice that $(F \cap I_2) \cup (G \cap I_2) = I_2$). Therefore $\widetilde{\Delta}^S_{\widehat{F}}(\al^{(1)}, \te^{(2)}) \subseteq A$. By %\cite[Theorem 2.5]{GMM} (bianchi), 
Theorem \ref{bianchi}, the level of $(\al^{(1)}, \te^{(2)})$ is strictly smaller than the level of $(\al^{(1)}, \al^{(2)})$. This proves $(\al^{(1)}, \al^{(2)}) \in A_{i+1}$.
\end{proof}

We define now a level function for all the elements of an arbitrary good semigroup $T$ which extend the notion of level for the elements not in the set $A$.

\begin{definition}
\label{virtuallevel}
\rm
Let $T$ be a good semigroup and let $A= \bigcup_{i=1}^N A_i$ be %its Apery set with respect to some nonzero element.
the complement of a good ideal. If $T$ is numerical,
$A= \{w_1, \ldots, w_N \}$ is finite and we set $A_i= \{w_i\}$.
We define a \it level function \rm
$\lambda: T \to \{1, \ldots, N+1\}$ in the following way:
\begin{itemize}
\item If $\al \in A_i$, $\lambda(\al)= i$.
\item If $\al \not \in A$, $\lambda(\al)= 1+ \max\{i \mbox{ such that } \al > \te \mbox{ for some } \te \in A_i \}$.
\end{itemize}
\end{definition}

\begin{theorem}
\label{thmnonlocallevels}
Let $S= S_1 \times S_2 $, $E \subsetneq S$, and $A$ be defined as above. Then,
given $(\al^{(1)}, \al^{(2)}) \in A$, the level of $(\al^{(1)}, \al^{(2)})$ in $A$ is equal to
$$ \lambda(\al^{(1)}) + \lambda(\al^{(2)}) -1.  $$
\end{theorem}

\begin{proof}
It is always true that $\boldsymbol{0}=(\boldsymbol{0}_{S_1}, \boldsymbol{0}_{S_2}) \in A_1$. Clearly $1 = \lambda(\boldsymbol{0}_{S_1}) + \lambda(\boldsymbol{0}_{S_1}) -1$ and the theorem is true for this element. We can thus work by iterating the following procedure: we assume the result true for an element of $A$ and  we prove it for another element consecutive to it in $A$. %induction on the total sum of the coordinates of the elements of $S$.
Observe that any two consecutive elements in $A$ are either of the form $(\al^{(1)}, \al^{(2)})$, $ (\al^{(1)}, \te^{(2)})$ with $\al^{(1)} \in A^{(1)}$ or $(\al^{(1)}, \al^{(2)})$, $ (\he^{(1)}, \al^{(2)})$ with $\al^{(2)} \in A^{(2)}$. Such elements are also consecutive in $S$.

Fix $\al^{(1)} \in A^{(1)}$ and consider $\al^{(2)} \in S_2$, $ \al^{(2)} \neq \boldsymbol{0}_{S_2}$. Let $\te^{(2)} \in A^{(2)}$ be a maximal element such that $\te^{(2)} < \al^{(2)}$ and $\lambda(\te^{(2)})$ is also maximal among the elements of $A^{(2)}$ smaller than $\al^{(2)}$. Pick a chain of consecutive elements of $S_2$ $$ \te^{(2)} \leq \te^{(2)}_1  \leq \cdots \leq \te^{(2)}_c \leq \al^{(2)} $$ and consider the corresponding chain of consecutive elements of $S$ 
$$ (\al^{(1)}, \te^{(2)}) \leq (\al^{(1)},\te^{(2)}_1)  \leq \cdots \leq (\al^{(1)},\te^{(2)}_c) \leq (\al^{(1)},\al^{(2)}). $$
By our choice of $\te^{(2)}$, we have $\te^{(2)}_1, \ldots, \te^{(2)}_c \in E_2$. An iterated application of Lemma \ref{directprod2} implies that $(\al^{(1)}, \al^{(2)})$ is in the same level of $(\al^{(1)}, \te^{(2)}_1)$ and we can therefore restrict to assume that $\te^{(2)}$ and $\al^{(2)}$ are consecutive.

Say that $(\al^{(1)}, \te^{(2)}) \in A_i$. Hence $(\al^{(1)}, \al^{(2)}) \in A_i \cup A_{i+1}$. 
By inductive hypothesis,  $i = \lambda(\al^{(1)}) + \lambda(\te^{(2)}) -1. $ We know that $\lambda(\al^{(2)}) \in \{ \lambda(\te^{(2)}), \lambda(\te^{(2)})+1\}$.
To conclude it is sufficient to prove that $\lambda(\al^{(2)}) = \lambda(\te^{(2)})$ if and only if $(\al^{(1)}, \al^{(2)}) \in A_i$. We apply Lemma \ref{directprod2}. Consider first the case $\al^{(2)} \in E_2$. By definition of $\lambda$ and by construction of $\te^{(2)}$, in this case $\lambda(\al^{(2)})= 1 + \lambda(\te^{(2)})$.
If there exists $\de^{(2)} \in E_2$ such that $\de^{(2)} \wedge \al^{(2)} = \te^{(2)}$, we would get $\te^{(2)} \in E_2$ that is a contradiction. This implies $(\al^{(1)}, \al^{(2)}) \in A_{i+1}$ by Lemma \ref{directprod2}. 

The other case to consider is when $\al^{(2)} \in A^{(2)}$. If $\al^{(2)} \gg \te^{(2)}$, then $\lambda(\al^{(2)})= 1 + \lambda(\te^{(2)})$ and clearly there cannot exists any element $\de^{(2)} \neq \te^{(2)} $ such that $\de^{(2)} \wedge \al^{(2)} = \te^{(2)}$. Since $\te^{(2)} \not \in E_2$, this again shows $(\al^{(1)}, \al^{(2)}) \in A_{i+1}$.
Otherwise $\al^{(2)} \in \Delta^{S_2}_{F}(\te^{(2)})$ for some non-empty set of indexes $F \supsetneq I_2$.
Now the existence of $\de^{(2)} \in E_2$ such that $\de^{(2)} \wedge \al^{(2)} = \te^{(2)}$ (which by Lemma \ref{directprod2} is equivalent to have $(\al^{(1)}, \al^{(2)}) \in A_{i}$) is equivalent to have $\widetilde{\Delta}^{E_2}_{\widehat{F}}( \te^{(2)}) \neq \emptyset$. The application of %\cite[Theorem 2.5]{GMM} and \cite[Theorem 2.8]{GMM} 
Theorems \ref{bianchi} and \ref{neri} shows that this is equivalent to have $\lambda(\al^{(2)})= \lambda(\te^{(2)})$ and conclude the proof.
\end{proof}

 %\bigskip

%\bf Descrizione dei livelli dell'Apery set nel caso non-locale a due rami, trasformare in un corollario \rm

As application we obtain a nice description of the levels of the Ap\'ery set of a non-local good semigroup in $\mathbb{N}^2$, see also Figure \ref{Figura 3}. A similar description can be obtained also in the case of good semigroup in $\mathbb{N}^d$ which splits completely as direct product of $d$ numerical semigroups.

\begin{cor}
\label{aperylocale}
Let $S = S_1 \times S_2 \subseteq \N^2$ be a non-local good semigroup and let $\om=(w_1, w_2) \in S$. Write $\Ap(S_1,w_1)=\lbrace u_{1}, u_{2}, \ldots, u_{w_1}  \rbrace$ and $\Ap(S_2,w_2)=\lbrace v_{1}, v_{2}, \ldots, v_{w_2}  \rbrace$ with the elements listed in increasing order. Set formally $u_{w_1+1}=v_{w_2+1}=\infty$. Then:
%Then $$ \Ap(S,\om)=\bigcup_{i=1}^{w_1} \Delta^S_1(u_i,0) \cup \bigcup_{j=1}^{w_2} \Delta^S_2(0,v_j).$$ Moreover, 
%Given $\al= (a,b) \in \Ap(S,\om)$:
%\begin{enumerate}
%\item If $a=u_i$ and $v_{j} < b < v_{j+1}$ or if $u_i < a < u_{i+1}$ and $b=v_j$, $\al$ is in the level $A_{i+j}$.
%\item If $u_i < a < a_{i+1}$ and $b=v_j$, $\al$ is in the level $A_{i+j}$.
%\item If $a=u_i$ and $b=v_j$, $\al$ is in the level $A_{i+j-1}$.
%\end{enumerate}
\begin{itemize}
    \item The level $A_1$ of $\Ap(S, \om)$ only consists of the element $(0,0)$.
    \item For $1 \leq i \leq w_1$, $1 \leq j \leq w_2$, the level $A_{i+j}$ of $\Ap(S, \om)$ is equal to the set 
$$ \lbrace (u_i, b) : v_{j} < b \leq v_{j+1} \rbrace \cup  \lbrace (a, v_j) : u_{i} < a \leq u_{i+1} \rbrace. $$
\end{itemize}
\end{cor}

\begin{figure}[H] 
  \centering
  \begin{tikzpicture}[scale=1.1]
	\begin{axis}[grid=major, xtick={0,4,7,11,14,18,21,24}, ytick={0,3,5,8,10,13}, yticklabel style={font=\tiny}, xticklabel style={font=\tiny}]
\addplot[only marks] coordinates{ (4,3) (4,6) (4,8) (4,9) (4,11) (4,12) (4,13) (8,3) (8,6) (8,8) (8,9) (8,11) (8,12) (8,13) (11,3) (11,6) (11,8) (11,9) (11,11) (11,12) (11,13) (12,3) (12,6) (12,8) (12,9) (12,11) (12,12) (12,13)
(15,3) (15,6) (15,8) (15,9) (15,11) (15,12) (15,13) (16,3) (16,6) (16,8) (16,9) (16,11) (16,12) (16,13)
(18,3) (18,6) (18,8) (18,9) (18,11) (18,12) (18,13) (19,3) (19,6) (19,8) (19,9) (19,11) (19,12) (19,13)
(20,3) (20,6) (20,8) (20,9) (20,11) (20,12) (20,13) (22,3) (22,6) (22,8) (22,9) (22,11) (22,12) (22,13) (23,3) (23,6) (23,8) (23,9) (23,11) (23,12) (23,13) (24,3) (24,6) (24,8) (24,9) (24,11) (24,12) (24,13)}; 
 \addplot[only marks, mark=text, mark options={scale=2,text mark={ \scriptsize 1}}, text mark as node=true] coordinates{(0,0)}; %1
 \addplot[only marks, mark=text, mark options={scale=2,text mark={ \scriptsize 2}}, text mark as node=true] coordinates{ (0,3) (0,5) (4,0) (7,0)}; %2
 \addplot[only marks, mark=text, mark options={scale=2,text mark={ \scriptsize 3}}, text mark as node=true] coordinates{ (0,6) (0,8) (0,9) (0,10) (4,5) (7,3) (7,5) (8,0) (11,0) (12,0) (14,0) }; %3
 \addplot[only marks, mark=text, mark options={scale=2,text mark={ \scriptsize 4}}, text mark as node=true] coordinates{ (0,11) (0,12) (0,13) (15,0) (16,0) (18,0) (19,0) (20,0) (21,0) (14,3) (14,5) (8,5) (11,5) (12,5) (4,10) (7,10) (7,6) (7,8) (7,9)}; %4
 \addplot[only marks, mark=text, mark options={scale=2,text mark={ \scriptsize 5}}, text mark as node=true] coordinates{ (22,0) (23,0) (24,0) (21,3) (21,5) (15,5) (16,5) (18,5) (19,5) (20,5) (7,11) (7,12) (7,13) (14,10) (8,10) (11,10) (12,10) (14,10) (14,6) (14,8) (14,9)} ; %5
 \addplot[only marks, mark=text, mark options={scale=2,text mark={ \scriptsize 6}}, text mark as node=true] coordinates{ (14,11) (14,12) (14,13) (21,10) (21,6)  (21,8) (21,9) (22,5) (23,5) (24,5) (21,10) (15,10) (16,10) (18,10) (19,10) (20,10)}; %6
 \addplot[only marks, mark=text, mark options={scale=1,text mark={ \scriptsize 7}}, text mark as node=true] coordinates{ (22,10) (23,10) (24,10) (21,11) (21,12) (21,13) }; %7
 \end{axis}
  \end{tikzpicture}
  \caption{\footnotesize The Ap\'ery set of the non-local good semigroup $S = S_1 \times S_2$ with respect to the element $\e= (4,3)$. Here, $S_1= \langle 4,7 \rangle$, $S_2= \langle 3,5 \rangle $, and $\Ap(S_1)= \{ 0,7,14,21 \}$, $\Ap(S_2)= \{ 0,5,10 \}$. The levels are indicated by distinct numbers. Black marks indicate the elements of $\e + S$.
   }
\label{Figura 3}
\end{figure}
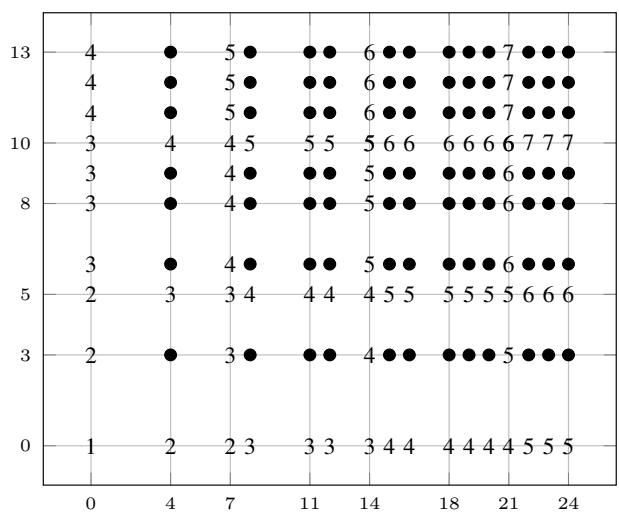

 \section{Well-behaved Ap\'ery sets}
 
 %\bf introdurre con un discorso sulle curve piane \rm
 
 In this section we consider a particularly nice class of complement of good ideals which includes the Ap\'ery sets of local value semigroups of plane curves. Value semigroups of plane curves with more then one branch have been considered in \cite{garcia}, \cite{c-d-gz}, \cite{waldi}, and their Ap\'ery set has been defined in the case of two branches in \cite{apery}. The definition of Ap\'ery set given in \cite{apery} is slightly different from that one given for general good semigroups in \cite{DGM} and \cite{GMM}, but they intuitively seem to agree on value semigroups of plane curves. We give an explicit proof that they coincide in Proposition \ref{aperysetcoincide}. 
 
 Let $S$ be the local value semigroup of a plane curve and let $A$ be the Ap\'ery set of $S$ with respect to some nonzero element $\om$. Then $A$ can be partitioned as $\bigcup_{i=1}^N A_i$ as defined in \cite{GMM} and recalled in Section 2 of this paper, or as $\bigcup_{i=1}^M B_i$ following the definition in \cite[Section 3]{apery}.
 
 Such definition can be summarized as follows. The set $B_M$ consists of all the maximal elements of $A$ with respect to the order relation $\ll$. The set $B_j$ for $j < M$ is defined inductively as the set of all maximal elements of $A \setminus (\bigcup_{i=j+1}^M B_i)$ with respect to the order relation $\ll$. In \cite{apery} is proved that $M$ is equal to the sum of the components of $\om$. Hence, by \cite[Theorem 4.4]{GMM}, $N=M$.

 Observing the specific case where $d=2$, we see that the main reason for which the two partitions of the Ap\'ery set coincide is due to the fact that, if $S$ is the value semigroup of a plane curve, then all the complete infimums of elements in the Ap\'ery set of $S$ are not in the Ap\'ery set (see Figure \ref{Figura 4} for a graphical interpretation of this property).
 
 \begin{figure}[H]
 \centering
 \tikzset{mark size=2}\begin{tikzpicture}[scale=1.1]
 \begin{axis}[grid=major,  ytick={0,3,5,9,14,19,24}, xtick={0,2,3,6,9,12,16}, yticklabel style={font=\tiny}, 
 xticklabel style={font=\tiny}]
 %\addplot[->, style=dotted, very thick]coordinates{(13.41935483870968,9)(16,9)};
 %\addplot[->, style=dotted, very thick]coordinates{(13.41935483870968,12)(16,12)};
 %\addplot[->, style=dotted, very thick]coordinates{(13.41935483870968,14)(16,14)};
 %\addplot[->, style=dotted, very thick]coordinates{(13.41935483870968,15)(16,15)};
 %\addplot[->, style=dotted, very thick]coordinates{(13.41935483870968,17)(16,17)};
 %\addplot[->, style=dotted, very thick]coordinates{(13.41935483870968,18)(16,18)};
 %\addplot[->, style=dotted, very thick]coordinates{(13.41935483870968,19)(16,19)};
 %\addplot[->, style=dotted, very thick]coordinates{(13.41935483870968,20)(16,20)};
 %\addplot[->, style=dotted, very thick]coordinates{(9,20.44444444444444)(9,24)};
 %\addplot[->, style=dotted, very thick]coordinates{(11,20.44444444444444)(11,24)};
 %\addplot[->, style=dotted, very thick]coordinates{(12,20.44444444444444)(12,24)};
 %\addplot[->, style=dotted, very thick]coordinates{(13,20.44444444444444)(13,24)};
 %\addplot [pattern = north east lines, draw=white]coordinates{(13,20)(16,20)(16,24)(13,24)(13,20)};
 \addplot[only marks] coordinates{(2,3)(4,6)(5,8)(6,9)(7,11)(8,12)(8,13)(9,12)(9,14)(10,12)(10,15)(10,16)(11,12)(11,15)(11,17)(11,18)(11,19)(11,20)(11,21)(11,22)(11,23)(11,24)(12,12)(12,15)(12,17)(12,18)(12,19)(13,12)(13,15)(13,17)(13,18)(13,20)(13,21)(13,22)(13,23)(13,24)(14,12)(14,15)(14,17)(14,18)(14,20)(14,21)(14,22)(14,23)(14,24)(15,12)(15,15)(15,17)(15,18)(15,20)(15,21)(15,22)(15,23)(15,24)(16,12)(16,15)(16,17)(16,18)(16,20)(16,21)(16,22)(16,23)(16,24)};
 \addplot[only marks, mark=text, mark options={scale=2,text mark={ \scriptsize 1}}, text mark as node=true] coordinates{(0,0)};
 \addplot[only marks, mark=text, mark options={scale=2,text mark={ \scriptsize 2}}, text mark as node=true] coordinates{(3,5)};
 \addplot[only marks, mark=text, mark options={scale=2,text mark={ \scriptsize 3}}, text mark as node=true] coordinates{(6,10)(7,9)(8,9)(9,9)(10,9)(11,9)(12,9)(13,9)(14,9)(15,9)(16,9)};
 \addplot[only marks, mark=text, mark options={scale=2,text mark={ \scriptsize 4}}, text mark as node=true] coordinates{(9,15)(9,16)(9,17)(9,18)(9,19)(9,20)(9,21)(9,22)(9,23)(9,24)(10,14)(11,14)(12,14)(13,14)(14,14)(15,14)(16,14)};
 \addplot[only marks, mark=text, mark options={scale=2,text mark={ \scriptsize 5}}, text mark as node=true] coordinates{(12,20)(12,21)(12,22)(12,23)(12,24)(13,19)(14,19)(15,19)(16,19)};\end{axis}\end{tikzpicture}
 \caption{\footnotesize The Ap\'ery set of the good semigroup $S$ associated to the plane curve $\K[[X,Y]]/((X^3-Y^2)\cdot(X^5-Y^3))\cong \K[[(t^2,u^3),(t^3,u^5)]]$ with respect to the element $\e= (2,3)$. Here, $S \subseteq S_1 \times S_2$ with
 %$S\subseteq S1\times S2\subseteq S$, 
 $S_1= \langle 2,3 \rangle$, $S_2= \langle 3,5 \rangle $, and $\Ap(S_1)= \{ 0,3 \}$, $\Ap(S_2)= \{0,5,10\}$. The levels are indicated by distinct numbers. Black marks indicate the elements of $\e + S$.}
 \label{Figura 4}
 \end{figure}
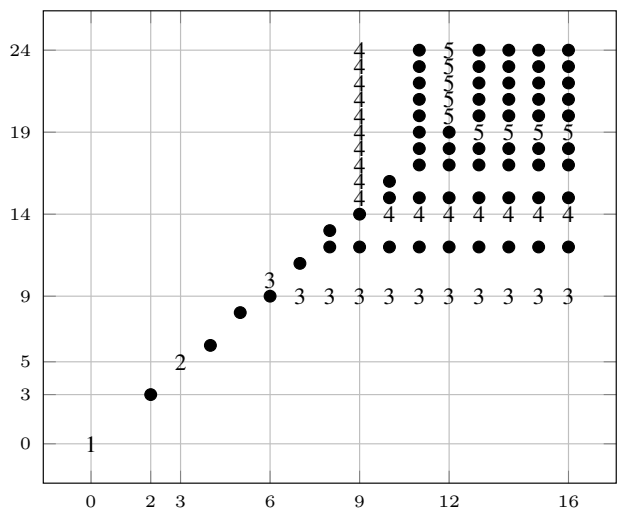

 This fact may be proved with valuation theoretic arguments following the notations and the method developed in \cite{apery}. In this paper, we rather prefer to give a proof using an approach based only on the combinatorics of good semigroups. 
 To do this, %show that Apery sets of value semigroups of plane curves are well-behaved sets, 
  we need to recall some key facts. %Let $S$ be the value semigroup of a plane curve and let $A$ be the Apery set of $S$ with respect to some nonzero element $\om$. Then $A$ can be partitioned as union $\bigcup_{i=1}^N A_i$ as recalled in Section 2 of this paper or as $\bigcup_{i=1}^M B_i$ following the definition in \cite[Section 3]{apery}. Such definition can be summarized as follows. The set $B_m$
  
  In \cite[Lemma 3.3]{apery} it is proved that if $\al \in B_i$ (and $i < M$), then there always exists $\be \in B_{i+1}$ such that $\be \gg \al$. 
  
  In \cite[Proposition 3.10]{apery} it is proved that if $\al \in B_i$, then $\Delta^S(\g + \om - \al) \subseteq B_{M-i+1}$.
  
  All the results in \cite{apery} are proved in the case $d=2$, but the authors mention explicitly in the introduction of the paper that each result until Theorem 4.1 (in particular all those in Section 3) can be proved with the same identical arguments for any number of branches $d$.
 Therefore, we can state and prove next proposition for arbitrary $d$. Recall that value semigroups of plane curves are symmetric, since rings of plane curves are always Gorenstein.

 \begin{prop}
\label{aperysetcoincide}
%Let $S=v(\mathcal{O})$ 
Let $S$ be the local value semigroup associated to a plane curve with $d$ branches. Write $\Ap(S)= \bigcup_{i=1}^N A_i = \bigcup_{i=1}^N B_i$, where $A_i$ and $B_i$ are the partitions defined respectively in \cite{GMM} and \cite{apery}. Then $A_i=B_i$ for every $i= 1,\ldots, N$. 
\end{prop}

\begin{proof}
First we show by decreasing induction on $i=N, \ldots, 1$ that $A_i \cap B_j = \emptyset$ for every $j < i$. By the definitions it is clear that $A_N=B_N= \Delta(\g + \om)$ (notice that since $S$ is local and symmetric, then $\g + \om \in \om + S$). Hence for every $j < N$, we get $A_N\cap  B_j= B_N \cap B_j = \emptyset$. Suppose now the result true for every $h > i$ and say that by way of contradiction there exists an element $\al \in A_i \cap B_j$ with $j < i$. By \cite[Lemma 3.3]{apery}, we can find $\be \in B_i$ such that $\be \gg \al$. Hence $\be \in A_h$ for some $h > i$. %(\cite[Remark 3.1]{DGM}). 
But by inductive hypothesis $A_h \cap B_i = \emptyset$, and this is a contradiction.

After setting this fact, assume to have 
$A_i \neq B_i$ for some $i$ and $A_j = B_j$ for every $j>i$. By considering the maximal elements with respect to $\ll$, after removing the levels $A_N, \ldots, A_{i+1}$, we get that necessarily there must exist %three elements: $\al \in A_{i-1} \cap B_i$, and $ \be, \de \in B_i \cap A_i$ such that $\al = \be \wedge \de$. Let now $\te$ be a minimal element in $\Delta^S(\g + \om - \al)$. By \cite[Proposition 3.10]{apery}, $\te \in B_{N-i+1}$, while by Proposition \ref{dualityprop2}, $\te \in A_{N-i+2}$. This contradicts the fact we proved in the first paragraph of this proof.
one element $\al \in A_{i-1} \cap B_i$ which is a complete infimum of elements in $A_i \cap B_i$.  Let now $\te$ be a minimal element in $\Delta^S(\g + \om - \al)$. By \cite[Proposition 3.10]{apery}, $\te \in B_{N-i+1}$, while by Proposition \ref{dualityprop2}, $\te \in A_{N-i+2}$. This contradicts the fact we proved in the first paragraph of this proof.
\end{proof}

 The following definition of well-behaved set aims to describe in a more general setting, and for an arbitrary number of branches, this specific behavior of the Ap\'ery sets of plane curves with respect to infimums.
 Through this section we describe properties of these well-behaved sets. An application to plane curves is discussed in the next section.

\begin{definition}
 \label{wellbehaved} \rm Let $S \subseteq \mathbb{N}^d$ be a good semigroup and let $A$ be the complement of a good ideal $E \subseteq S$. We say that $A$ is \it well-behaved \rm if whenever $\al = \be^{(1)} \wt \cdots \wt \be^{(r)}$ with $\be^{(j)} \in \widetilde{\Delta}^S_{G_j}(\al) \subseteq A$ for every $j$, then $\al \in E$.
 \end{definition}
 
 If $d=2$, this corresponds to say that whenever $\Delta^S(\al) \subseteq A$ and it is non-empty, then $\al \not \in A$. %The next property is equivalent: \bf tenerla qui? \rm

 %\begin{definition} \bf A 2 RAMI \rm
%Let $S$ be a good semigroup and let $A=\bigcup_{i=1}^e A_i = S\setminus E$, for $E$ a good ideal of $S$. We say that $A$ is \it well-behaved \rm if for every $i=1, \ldots, e$ and every $\al, \be \in A_i$, $\al \wedge \be \in A$ if and only if $\al \wedge \be \in \{ \al, \be \}$.
%\end{definition}
 
 %\bf il prossimo lo riscriverei usando solo le ipotesi che ci servono e dicendo il perché le curve piane le soddisfano \rm
 
 %Therefore, 
 
 The Ap\'ery sets of value semigroups of plane curves are well-behaved as consequence of next proposition. Again we use the the fact proved in \cite[Lemma 3.3]{apery} and recalled previously.

 \begin{prop}
 \label{curvepianesonowellbehaved}
 %Let $S$ be the value semigroup of a plane curve and let $\om$ be a nonzero element of $S$. Then $A= \Ap(S, \om)$ is well-behaved.
 Let $S$ be a good semigroup and let $A= \bigcup_{i=1}^N A_i$ be the complement of a good ideal $E \subseteq S$. Suppose that for every $i= 1, \ldots, N-1$ and for every $\al \in A_i$, there exists $\be \in A_{i+1}$ such that $\be \gg \al$. Then $A$ is well-behaved.
 \end{prop}
 
 \begin{proof}
 Pick $\al \in A_i$ and assume by way of contradiction $\al = \be^{(1)} \wt \cdots \wt \be^{(r)}$ with $\be^{(j)} \in \widetilde{\Delta}^S_{G_j}(\al) \subseteq A$ for every $j$. This assumption allows to further assume that all the $\be^{(j)}$ are consecutive to $\al$ and at least one of them is in $A_{i+1}$. %Being $S$ the value semigroup of a plane curve, as consequence of \cite[Lemma 3.3]{apery}, we can find 
 By hypothesis we can find $\be \in A_{i+1}$ such that $\be \gg \al$. For every $j$, since $\al < \be \wedge \be^{(j)} \leq \be^{(j)}$, then $\be \wedge \be^{(j)} = \be^{(j)}$ implying $\be > \be^{(j)}$. Choose now $j$ such that $\be^{(j)} \in A_{i+1}$.
 Hence, $\be \not \gg \be^{(j)}$ and $\be \in \Delta^S_{F_j}(\be^{(j)})$ with $G_j \subseteq \wh F_j$.  It follows that 
 $ \widetilde{\Delta}^S_{\wh F_j}(\be^{(j)}) \subseteq  \widetilde{\Delta}^S_{G_j}(\al) \subseteq A $. Theorem \ref{bianchi} applied to $\be$ and $\be^{(j)}$ forces $\be^{(j)}$ to be in a level strictly smaller than the level of $\be$. This is a contradiction. % since at least one of the $\be^{(j)}$ is in $A_{i+1}$.
 \end{proof}
 
 Next lemma shows that, if $A$ is well-behaved, then any subspace of $\mathbb{N}^d$ whose intersection with $S$ is non-empty and contained in $A$, has to be contained in a unique level.
 
 \begin{lemma}
 \label{wellbehavedlem1}
 Let $S \subseteq \mathbb{N}^d$ be a good semigroup and let $A$ be the complement of a good ideal $E \subseteq S$. Suppose $A$ to be well-behaved.
 Let $\om \in \mathbb{Z}^d$ be any element and assume $\Delta^S_F(\om) \subseteq A$ (and it is non-empty). Then $\Delta^S_F(\om) \subseteq A_i$ for some $i$.
 \end{lemma}
 
 \begin{proof}
 It sufficient to show that any two consecutive elements in $\al, \be \in \Delta^S_F(\om)$ are in the same level.
 Say that $\be \in \Delta^S_H(\al)$ with $H \supseteq F$ and suppose $\al \in A_i$ and $\be \in A_{i+1}$. This implies that $\widetilde{\Delta}^E_{\wh H}(\al) = \emptyset$ otherwise we would get a contradiction with Theorem \ref{neri}. Using property (G2) we can write $\al = \be \wt \be^{(2)} \wt \cdots \wt \be^{(r)}$ with, for $j=2, \ldots, r$, $\be^{(j)} \in \widetilde{\Delta}^S_{G_j}(\al)$ and $G_j \supseteq \wh H$. Hence, 
 $\widetilde{\Delta}^S_{G_j}(\al) \subseteq \widetilde{\Delta}^S_{\wh H}(\al) \subseteq A$. By Definition \ref{wellbehaved}, we get a contradiction since $\al \in A$.
 \end{proof}
 
 As consequence of Theorem \ref{bianchi} we also have that: 
\begin{cor}
\label{wellbehavedcor}
With the same notation of Lemma \ref{wellbehavedlem1}, %suppose that $\Delta^S_F(\om), \Delta^S_G(\om) \subseteq A$ for two sets $F,G \neq \emptyset$ such that $F \cup G = I$ (and they are non-empty). Then there exist $i$ such that $\Delta^S_F(\om) \cup \Delta^S_G(\om) \subseteq A_i$. % for every nonempty set $F \subsetneq I$.
%In particular, 
if $d=2$ and $\Delta^S(\om)$ is non-empty and contained in $A$, then $\Delta^S(\om) \subseteq A_i$ for some $i$.
\end{cor}
 
 As main consequence of the previous lemma, 
 for $A= S \setminus E$ a well-behaved set, we give a criterion describing the level of all the elements having some coordinate not in the projection of $E$. Our notation is similar to that used in Section \ref{nonlocal} in the case of non-local good semigroups. 
 
 Let $S \subseteq \mathbb{N}^d$ be a good semigroup and let $d_1, d_2 \geq 1$ be positive integers such that $d_1 + d_2 = d$. Write a partition of the set of indexes $I= \lbrace 1, \ldots, d \rbrace = I_1 \cup I_2 $, with $|I_j| = d_j$.
 % and without loss of generality suppose $I_1=\lbrace 1, \ldots, d_1 \rbrace $, $I_2= \lbrace d_1+1, \ldots, d_1+ d_2 \rbrace$. 
 Define $S_1$ to be the canonical projection of $S$ on the set of indexes $I_1$ and $S_2$ to be the canonical projection of $S$ on the set of indexes $I_1$. The set $S_1$ and $S_2$ are also good semigroups $S \subseteq S_1 \times S_2$. For each element $\al \in S$, we write $\al = (\al^{(1)}, \al^{(2)})$ with $\al^{(j)} \in S_j$ and we use the corresponding notation also for $\al \in \mathbb{N}^d \subseteq \mathbb{N}^{d_1} \times \mathbb{N}^{d_2}$. Let $E_1$, $E_2$ be the projections of $E$ and $A^{(h)}:= S_h \setminus E_h$ for $h=1,2$. We denote by $A^{(h)}_i$ the $i$-th level of the partition of $A^{(h)}$ in the semigroup $S_h$.
 
 %Without loss of generality we will generally assume $I_1=\lbrace 1, \ldots, d_1 \rbrace $, $I_2= \lbrace d_1+1, \ldots, d_1+ d_2 \rbrace$. \bf forse non serve \rm

%Let $d_1, d_2 \geq 1$ be positive integers.
%Let $S_1 \subseteq \mathbb{N}^{d_1} $, $S_2 \subseteq \mathbb{N}^{d_2}$ be two good semigroups, not necessarily local. Consider a local good semigroup $S \subseteq S_1 \times S_2  \subseteq \mathbb{N}^d$ where $d= d_1+d_2$.
%Pick $\e= (\e^{(1)}, \e^{(2)}) \in S$ and assume $(\e^{(1)}, \e^{(2)}) \gg (\boldsymbol{0}, \boldsymbol{0})$.
%Let $A:=\Ap(S,\e )$ and call $\Ap(S_1):= \Ap(S_1, \e^{(1)})$, $\Ap(S_2):= \Ap(S_2, \e^{(2)})$
%We want to describe the level of $A$ in term of the levels of $\Ap(S_1)$ and $\Ap(S_2)$. It follows easily by the definitions that $$ \Ap(S)= (\Ap(S_1) \times S_2) \cup (S_1 \times \Ap(S_2)).  $$
%Setting $E_1:=S_1\setminus \Ap(S_1)$ and $E_2:=S_2\setminus \Ap(S_2)$, we get 
%$$  E := S \setminus A = E_1 \times E_2.   $$
%After writing each element of $S$ in the form $$(\al, \be)= (\alpha_1, \ldots, \alpha_{d_1}, \beta_{d_1+1} , \ldots, \beta_{d_1+d_2}),$$
 %it will be useful to define the set of indexes $I_1:= \lbrace 1, \ldots, d_1 \rbrace $, $I_2:= \lbrace d_1+1, \ldots, d_1+d_2 \rbrace $, and $I:= I_1 \cup I_2$.
 
 \begin{theorem}
 \label{levelswellbehaved}
 Let $S \subseteq \mathbb{N}^d$ be a local good semigroup and let $A$ be the complement of a good ideal $E \subsetneq S$. Suppose $A$ to be well-behaved. Define $S_1, S_2$ as above. Pick $\al^{(1)} \in A^{(1)}_i$. Then
 % $$ \Gamma(\al^{(1)}) :=  \{ (\al^{(1)}, \al^{(2)})  :  \al^{(2)} \in S_2   \} \cap S \subseteq A_i. $$
 $$ \Gamma(\al^{(1)}) :=  \{ \al \in \mathbb{N}^d  :  \al^{(2)} \in S_2   \} \cap S \subseteq A_i. $$
 The analogous result holds for elements in the projection $S_2$ by switching the coordinates.
 \end{theorem}
 
 \begin{proof}
 Since $S$ is local, $A_1= \{ \boldsymbol{0} \} $ and therefore the result is true for 
 $\al^{(1)} = \boldsymbol{0}_{S_1}$. Thus we can reduce to prove that, if $\be^{(1)} < \al^{(1)}$ are consecutive in $A^{(1)}$ and the thesis is true for $\be^{(1)}$, then it is true for $\al^{(1)}$. Clearly, for $\al^{(1)} > \boldsymbol{0}_{S_1}$, $\Gamma(\al^{(1)}) = \Delta^S_{I_1}(\al^{(1)}, \boldsymbol{0}_{S_2}) \subseteq A$. By Lemma \ref{wellbehavedlem1}, $\Gamma(\al^{(1)}) \subseteq A_h$ for some $h$. To show that $h=i$ it is sufficient to show that the minimal element of $\Gamma(\al^{(1)})$, that we call $\al$, is in $A_i$. We consider now different cases:
 
\noindent \bf Case 1: \rm $\be^{(1)}, \al^{(1)}$ are consecutive in $S_1$. \\
 Let $\be'$ be the minimal element of $\Gamma(\be^{(1)})$. If $\be' < \de < \al$, then since $\be^{(1)}, \al^{(1)}$ are consecutive in $S_1$ and by minimality of $\al^{(1)}$, necessarily $\de^{(1)}=\be^{(1)}$. Hence we can find an element $\be \in \Gamma(\be^{(1)})$ such that $\be$ and $\al$ are consecutive in $S$. Say that $\be^{(1)} \in A^{(1)}_k$ with $k \in \{i-1, i \}$. Assuming by induction the thesis true for $\be^{(1)}$, we get 
 $\be', \be \in A_k$.  % Say that $\al^{(1)} \in \Delta^{S_1}_F(\be^{(1)})$, with possibly $F= \emptyset$. Then $\al \in \Delta^S(\be)$
 Say that $\al \in \Delta^S_F(\be)$ with $F \nsupseteq I_1$ and possibly $F= \emptyset$. Then $\al^{(1)} \in \Delta^{S_1}_{F \cap I_1}(\be^{(1)})$. Notice that $h \in \{k, k+1 \}$. We have to show that $k= i-1$ if and only if $h= k+1$. This will imply $h=i$. 
 
 First suppose $k= i-1$. Then either $F \cap I_1 = \emptyset$ (i.e. $\be^{(1)} \ll \al^{(1)}$) or, by Theorem \ref{neri}, $\widetilde{\Delta}^{S_1}_{\wh F \cap I_1}(\be^{(1)}) \subseteq A^{(1)}.$ If $F \cap I_1 \neq \emptyset$ this implies that $\widetilde{\Delta}^{S}_{\wh F}(\be) \subseteq A,$ since $E_1$ is the projection of $E$. If instead, $F \cap I_1 = \emptyset$, if also $F = \emptyset$ we are done since $\be \ll \al$, otherwise $F \subseteq I_2$. In the case $F \subseteq I_2$, it follows that $$\widetilde{\Delta}^{S}_{\wh F}(\be)  \subseteq \widetilde{\Delta}^{S}_{I_1}(\be) \subseteq \Gamma(\be^{(1)}) \subseteq A.$$ In all such cases, by Theorem \ref{bianchi}, $\be$ is in a level strictly smaller than $\al$, implying $h= k+1= i$.
 
 Conversely, suppose $k=i$ and show $h=i$. In this case clearly $F \cap I_1 \neq \emptyset$ and by Theorem \ref{bianchi} $\widetilde{\Delta}^{E_1}_{\wh F \cap I_1}(\be^{(1)}) \neq \emptyset.$ Pick $\de^{(1)} \in \widetilde{\Delta}^{E_1}_{\wh F \cap I_1}(\be^{(1)})$ and let $\de$ be the minimal element of $\Gamma(\de^{(1)}) \cap E$. By minimality of $\be'$, since $\be^{(1)} < \de^{(1)}$, then $\be' < \de$. Thus $\de \in \Delta^S_G(\be')$ with $G \cap I_1 \supseteq \wh F \cap I_1$.  Applying property (G2) to $\de$ and $\be'$ we write
 $\be' = \de \wt \om^1 \wt \cdots \wt \om^r$ with $\om^j \in \Delta^S_{H_j}(\be)$ and $\bigcap_{j=1}^r H_j = \wh G$. As usual we can assume that $\om^j$ are consecutive to $\be'$ and by Theorem \ref{neri}, $\om^j \in A_i$.
 We prove that there exists $j$ such that $H_j \cup F \nsupseteq I_1$. Indeed, if for every $j=1, \ldots, r$, $I_1 \subseteq H_j \cup F$, we would have $I_1 \subseteq (\bigcap_{j=1}^r H_j \cup F) = \wh G \cup F.  $ Hence $$ I_1 = I_1 \cap (\wh G \cup F) = (I_1 \cap \wh G) \cup (I_1 \cap F) = (I_1 \cap F)$$ and this is a contradiction since $F \nsupseteq I_1$. Choose now this $j$ such that $H_j \cup F \nsupseteq I_1$ and call $\om := \om^j$ and $H:=H_j$.
 Since $\be' \leq \be < \al$, $\al \in \Delta^S_U(\be')$ with $U \subseteq F$. Furthermore, since $I_1 \nsubseteq H \cup F$, then $I_1 \nsubseteq H \cup U$, and $H \cup U \neq I$.  
 Observing that $\om \wedge \al \in \Delta^S_{H \cup U}(\be')$ we get that $\be' < \om \wedge \al \leq \al$. But the condition $I_1 \nsubseteq H \cup U$ implies that $\om \wedge \al \not \in \Gamma(\be^{(1)}).$ Thus, since $\be^{(1)}, \al^{(1)}$ are consecutive in $S_1$, $(\om \wedge \al)^{(1)}= \al^{(1)}$, and by minimality of $\al$ in $\Gamma(\al^{(1)})$ we get
 $\om \wedge \al = \al$. This implies $\om \geq \al$ and $\al \in A_i$.
 
 \noindent \bf Case 2: \rm There exists $\te^{(1)} \in E_1 $ such that $\be^{(1)} < \te^{(1)}< \al^{(1)}$. \\
 Denote by $\be$ the minimal element of of $\Gamma(\be^{(1)})$ and assume the thesis true for $\be^{(1)}$. By Lemma \ref{neroinmezzo}, $\be^{(1)} \in A_{i-1}^{(1)}$ and hence $\be \in A_{i-1}$. Say again that $\al \in \Delta^S_F(\be)$ with $F \nsupseteq I_1$ and possibly $F= \emptyset$.
 Thus $\al^{(1)} \in \Delta^{S_1}_{F \cap I_1}(\be^{(1)})$ and $\te^{(1)} \in \Delta^{S_1}_{V \cap I_1}(\be^{(1)})$ with $V \supseteq F$. %If $F \cap I_1 = \emptyset$, then $\al \gg \be$ and the thesis follows.
 If $F \cap I_1 \neq \emptyset$, using that $\be^{(1)} \in A^{(1)}$, by property (G1) of the good ideal $E_1$, we get $\widetilde{\Delta}^{E_1}_{\wh V \cap I_1}(\be^{(1)}) = \emptyset$. Therefore,
 $ \widetilde{\Delta}^{S_1}_{\wh F \cap I_1}(\be^{(1)}) \subseteq \widetilde{\Delta}^{S_1}_{\wh V \cap I_1}(\be^{(1)}) \subseteq A^{(1)}.  $ Now both in this case and in the case when $F \cap I_1 = \emptyset$ we conclude proceeding as in Case 1, in the subcase $k=i-1$.
 \end{proof}
 
 %\bf il prossimo è lo stesso risultato a due rami, mettere come corollario? \rm 
 
 As a corollary we describe the case $d=2$ adding an observation on upper bounds for the maximal coordinate of elements in the first levels. Again the reader can compare this statement with the representation in Figure \ref{Figura 4}.
% \begin{definition}
%Let $S$ be a good semigroup and let $A=\bigcup_{i=1}^e A_i = S\setminus E$, for $E$ a good ideal of $S$. 
%We say that a level $A_i$ is \it bounded with respect to the coordinate $h$ \rm if there exists $n \in \mathbb{N}$ such that for every $\al=(a_1, a_2) \in A_i$, $a_h < n$. %In the opposite case we say that $A_i$ is \it infinite with respect to the coordinate $j$. \rm
%\end{definition}

\begin{cor}
\label{wellbehavedproj}
Let $S \subseteq S_1 \times S_2 \subseteq \mathbb{N}^2$ be a local good semigroup and pick $\om= (\omega_1, \omega_2) \in S$. Suppose $A= \Ap(S, \om)$ to be well-behaved. Let $\lbrace u_1, \ldots, u_{\omega_1} \rbrace$ be the Ap\'ery set of $S_1$ with respect to $\omega_1$. Then for every $i=1, \ldots \omega_1$, $$ \Delta_1^S(u_i, -1) \subseteq A_i. $$ Moreover if $\al=(a_1, a_2) \in A_i$, then $a_1 \leq u_i $. % and $A_i$ is bounded with respect to the coordinate $1$.
The analogous result holds for the projection $S_2$ by switching the coordinates.
\end{cor}

\begin{proof}
Clearly the set $\Delta_1^S(u_i, -1)$ is non-empty and by Lemma \ref{wellbehavedlem1} is contained in the level $A_i$. Call $\be=(u_i, b_2)$ the minimal element of $\Delta_1^S(u_i, -1)$.
Assume there exists $\al \in A_i$ such that $a_1 > u_i$. Then $a_2 \leq b_2$ since $\al \not \gg \be$.
By property (G1), the minimality of $\be$ excludes the case $a_2 < b_2$. If $a_2 = b_2$, by property (G2) we find $\te \in \Delta^S_1(\be) \subseteq A_i$. Hence $\be = \al \wedge \te$ and this is a contradiction since they are all elements of $A_i$.
\end{proof}

 We conclude this section with some further results in the case when $d=2$, which will be needed to prove the main theorem of next section. First we have a lemma describing an equivalent condition of being well-behaved.
 
  \begin{lemma}
 \label{2ramiwb}
 Let $S \subseteq \mathbb{N}^2$ and let $A$ be the complement of a proper good ideal of $S$. The following assertions are equivalent:
 \begin{enumerate}
     \item $A$ is well-behaved.
     \item For every level $A_i$ and every $\al, \be \in A_i$, %such that $\al \wedge \be \neq \al, \be$ then 
     $\al \wedge \be \in A$ if and only if $\al \wedge \be \in \{ \al, \be \}$.
     \item For every level $A_i$ with $i < N$ and every $\al \in A_i$, there exists $\be \in A_{i+1}$ such that $\be \gg \al$.
 \end{enumerate}
 \end{lemma}
 
 \begin{proof}
  1. $\Rightarrow$ 2.  Let $\al, \be \in A_i$, $\te = \al \wedge \be$ with $\al \in \Delta_1^S(\te)$ and $\be \in \Delta_2^S(\te)$. Suppose $\te \neq \al, \be$. If $\Delta^S(\te) \subseteq A$ we are done. Hence assume $\Delta^E_1(\te) \neq \emptyset$ (if $\Delta^E_2(\te) \neq \emptyset$ we use the analogous argument) and $\te \in A$. Therefore, $\te \in A_h$ with $h < i$. By property (G1) used on $E$ we get $\Delta^S_2(\te) \subseteq A$ and by Theorem \ref{neri} we can find $\om \in A_h \cap \Delta^S_2(\te)$. We can suppose $\om$ to be the maximal element in $A_h$ such that $\te < \om < \be $. It follows that $\Delta^S_1(\om) \subseteq A$ otherwise we would contradict such maximality by Theorem \ref{neri}. Thus $\Delta^S(\om) \subseteq A$ and it is non-empty. %By property (G2) applied to $\be$ and $\om$ we find a contradiction with 
 This contradicts the hypothesis of having $A$ well-behaved.
 
 2. $\Rightarrow$ 3. Suppose there exists $\al \in A_i$ with $i <N$ that is not dominated by any element of $A_{i+1}$. Necessarily, by definition of the partition in levels, $\al = \te \wt \om$ with $\te, \om \in A_{i+1}$. This contradicts 2.
 
 3. $\Rightarrow$ 1.
 This is Lemma \ref{curvepianesonowellbehaved}.
 %Let $\al \in A_h$ and assume $\Delta^S(\al) \subseteq A$ and it is non-empty. By Theorem \ref{bianchi} there exists $\de,\te \in A_i$ with $i > h$ such that $\de \wedge \te = \al$. A contradiction.
 \end{proof}

 Next result gives strong restrictions on the areas of $\mathbb{N}^2$ where the elements of a fixed level of a well-behaved set can exist. This description is done in term of the absolute elements of $S$ which are also in that level. We recall that $\al \in S $ is an absolute element if $\Delta^S(\al)= \emptyset.$ 
 
 In general, we say that a level $A_i$ is \it bounded with respect to the coordinate $h$ \rm if there exists $n \in \mathbb{N}$ such that for every $\al=(a_1, a_2) \in A_i$, $a_h < n$. In the opposite case we say that $A_i$ is \it infinite with respect to the coordinate $h$. \rm
 
 Let $c_E= (q_1, q_2)$ be the conductor of $E:=S \setminus A$.
 Fixed a level $A_i$, let $\{\te^{(1)}, \ldots, \te^{(r)} \}$ be all the absolute element of $S$ in the level $A_i$. We assume them to be ordered increasingly with respect to the first component (thus decreasingly with respect to the second one).
 Moreover, if $A_i$ is infinite with respect to the coordinate $1$, we define $\te^{(0)}:=(q_1, s_2)$ such that $\Delta_1(\te^{(0)}) \subseteq A_i$. Similarly, if $A_i$ is infinite with respect to the coordinate $2$, define $\te^{(r+1)}=(s_1, q_2 )$ such that $\Delta_2(\te^{(r+1})) \subseteq A_i$.\\
The following proposition describes the structure of the levels in a well-behaved set, in a good semigroup $S\subseteq \N^2$ (see Figure \ref{Figura 5} for a graphical representation).
 
 \begin{prop}
\label{structurelevel}
Let $A \subseteq S \subseteq \mathbb{N}^2$ be a well-behaved set. Let $\te^{(k)}$ for $k=0, \ldots, r+1$ be defined as above for a fixed level $A_i$.
%Fix $i$ and let $\{\te^{(1)}, \ldots, \te^{(r)} \}$ be all the absolute element of $S$ in the level $A_i$. Assume them to be ordered increasingly with respect to the first component (thus decreasingly with respect to the second one).
%Moreover, if $A_i$ is infinite with respect to the coordinate $1$, let $\te^{(0)}=(\gamma_1 + e_1 +1, s_2)$ be  such that $\Delta_1(\te^{(0)}) \subseteq A_i$. Similarly, if $A_i$ is infinite with respect to the coordinate $2$, let $\te^{(r+1)}=(s_1, \gamma_2 + e_2 +1)$ be such that $\Delta_2(\te^{(r+1})) \subseteq A_i$. 
Let $\al=(a_1,a_2) \in A_i$. Then one of the following assertions holds:
\begin{itemize}
    \item[i)] $\al \in \Delta_1(\te^{(0)}) \cup \Delta_2(\te^{(r+1)})$.
    \item[ii)] There exists $k \in \{0, \ldots, r\}$ such that $ \te^{(k)} \leq \al <  \te^{(k)} \wedge \te^{(k+1)}$.
    \item[iii)] There exists $k \in \{1, \ldots, r+1 \}$ such that $\te^{(k-1)} \wedge \te^{(k)} < \al \leq \te^{(k)}$.
    \item[iv)] $A_i$ is bounded with respect to $1$ and $\te^{(r)} \in \Delta_1^S(\al)$.
    \item[v)] $A_i$ is bounded with respect to $2$ and $\te^{(1)} \in \Delta_2^S(\al)$.
\end{itemize}
In particular $\al$ shares at least one coordinate with some $\te^{(k)}$.
\end{prop}

\begin{figure}[H]
\centering
	\pgfplotsset{ticks=none}
\tikzset{mark size=1}
\begin{tikzpicture}[scale=1]
	\begin{axis}[ xmin=0, ymin=0, xmax=20, ymax=20, ]

%	\addplot[only marks,color=red] coordinates{(4,13)(4,14)(4,15)(4,16)(4,17)(4,18)(4,19)(4,20)}; 

  \addplot[only marks, mark=x,mark options={scale=2}] coordinates{(2,13)(6,8)(12,5)};
  \addplot[only marks, mark=o,mark options={scale=2}] coordinates{(2,14)(2,15)(2,17)(2,18)(2,19)};
  \addplot[only marks, mark=o,mark options={scale=2}] coordinates{(3,13)(5,13)};
  \addplot[only marks, mark=o,mark options={scale=2}] coordinates{(6,10)(6,11)};
  \addplot[only marks, mark=o,mark options={scale=2}] coordinates{(7,8)(9,8)};   
  \addplot[only marks, mark=o,mark options={scale=2}] coordinates{(12,6)}; 
  \addplot[only marks, mark=o,mark options={scale=2}] coordinates{(12,7)(14,5)(16,5)};  
  \addplot[only marks, mark=o,mark options={scale=2}] coordinates{(17,3)};  
        \addplot[only marks, mark=*,mark options={scale=1.5}] coordinates{(2,16)(6,13)(12,8)(17,5)};
	\node[font=\footnotesize] at (3.2,16.8) {$\bs{\theta^{(0)}}$};
	\node[font=\footnotesize] at (7,13.8) {$\bs{\theta^{(1)}}$};
%	\node[font=\tiny] at (7,8.8) {$\bs{\theta^{(1)}}\wedge\bs{\theta^{(2)}}$};
    \node[font=\footnotesize] at (13,8.8) {$\bs{\theta^{(2)}}$};
    \node[font=\footnotesize] at (18,5.8) {$\bs{\theta^{(3)}}$};
\end{axis}
\end{tikzpicture}
\caption{\footnotesize In this figure is represented the level $A_i$ of a well-behaved set $A$ in a good semigroup $S\subseteq \N^2$, as described by Proposition \ref{structurelevel}. The elements $\{\theta^{(0)},\theta^{(1)},\theta^{(2)},\theta^{(3)}\}$ are denoted by $\bullet$, all other elements of $A_i$ are denoted by $\circ$. The elements marked by $\times$ are not in the level $A_i$. } 
\label{Figura 5}
\end{figure}
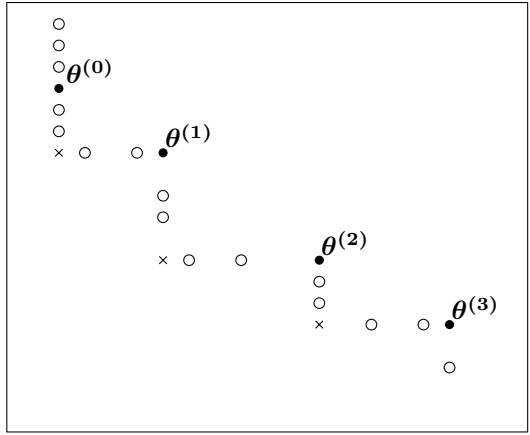

\begin{proof}
%Pick $\al=(a_1,a_2) \in A_i$, $\al \neq \te^{(k)}$ for every $k$. Also 
If $a_h > q_h$ for some $h=1,2$, we are necessarily in the situation described in the first item (see \cite[Lemma 1, items 7-8]{DGM}).
Thus we can restrict to assume $a_h < q_h$ in both coordinates and $\al \neq \te^{(k)}$ for every $k$. First consider the case $\theta^{(k)}_1 \leq a_1 < \theta^{(k+1)}_1$. Since two distinct elements of the same level are incomparable with respect to the order relation $\ll$, we must have $\theta^{(k)}_2 \geq a_2 \geq \theta^{(k+1)}_2$. Furthermore, since $A$ is well-behaved, by Lemma \ref{2ramiwb}, $\de:= \te^{(k+1)} \wedge \te^{(k)} \in E $. We only have to exclude the case $\al \gg \de$. For this we can clearly suppose $\al$ to be a maximal element in $A_i$ such that $\al \gg \de$. Indeed, since $\te^{(k+1)} , \te^{(k)}$ are absolute elements in $A_i$, there cannot exist infinitely many element $\al \in A_i$ such that $\al \gg \de$. But $\al$ cannot be another absolute element in $A_i$. Thus $\Delta^S(\al) \neq \emptyset$. If $\Delta^E(\al) \neq \emptyset$ we contradict the maximality of $\al$ using Theorem \ref{neri}. If instead $\Delta^E(\al) = \emptyset$, we contradict the hypothesis of having $A$ well-behaved. 

The only case we still need to discuss is when $a_1 < \theta^{(1)}_1$ and $A_i$ is bounded with respect to the second coordinate (the other cases are obtained by analogy switching the coordinates). Also in this case, if $a_2 > \theta^{(1)}_2$, since $A_i$ is bounded, we may say that $\al$ is maximal with respect to satisfy such property. The same argument as above forces $\al$ to be a new absolute element, which is a contradiction.
\end{proof}

\section{An application to plane curves}

%\bf i livelli nostri coincidono con quelli di D'Anna-Barucci-Froberg nel caso di curve piane, aggiungere definizioni, la prossima proposizione può essere enunciata a più rami? \rm

%\bf questo dobbiamo decidere se tenerlo qui a 2 rami o metterlo prima \rm

In this section we prove a result for value semigroups of plane curves with two branches which extends \cite[Theorem 4.1]{apery} in the non-local case and provides an alternative method of reconstructing the value semigroup from the multiplicity sequence of the curve. %We give a brief overview of this theory. 

%\medskip

The setting is the following.
Let $S=v(\mathcal{O}) \subseteq S_1 \times S_2$ be the local value semigroup of a plane curve and $S_1$ and $ S_2$ be its numerical projections. Clearly $S_1$ and $ S_2$ are value semigroups of plane branches.
 Let $\e=(e_1, e_2)$ be the minimal nonzero element of $S$ and let $A= \bigcup_{i=1}^e A_i$ be the Ap\'ery set with respect to $\e$, where $e=e_1+e_2$. 

Let $S'$ be the value semigroup of the blow up of $\mathcal{O}$ and suppose that $S'$ is not local. 
In this case $S' = S_1'\times S_2'$ and where $S_1'$ and $S_2'$ are the respective blowups of $S_1$ and $S_2$.

We recall that all the above semigroups %$S, S', S_1, S_2, S_1', S_2'$ 
are value semigroups of some plane curve, hence they are symmetric and their Ap\'ery sets are well-behaved as consequence of Proposition \ref{curvepianesonowellbehaved}.

By Ap\'ery's Theorem in the case of plane branches,
 after setting $\Ap(S_1)= \{u_{1}, u_{2}, \ldots, u_{e_1}\}$ and $\Ap(S_2)= \{v_{1}, v_{2}, \ldots, v_{e_2}\}$, we get $$\Ap(S_1', e_1)= \{u_{1}, u_{2}-e_1, u_3-2e_1, \ldots, u_{e_1}-(e_1-1)e_1\},$$ $$\Ap(S_2', e_2)= \{v_{1}, v_{2}-e_2, v_3-2e_2, \ldots, v_{e_2}-(e_2-1)e_2 \}.$$
Let $A'= \bigcup_{i=1}^E A_i'$ be the Ap\'ery set of $S'$ with respect to $\e$.
The levels of the Ap\'ery set of $S'$ are described in Theorem \ref{thmnonlocallevels} and Corollary \ref{aperylocale}. 
We aim to prove next theorem:
\begin{theorem}
\label{teoremaaperynonlocale}
Let $S=v(\mathcal{O}) \subseteq S_1 \times S_2$ be a local value semigroup of a plane curve and suppose $S'= S_1'\times S_2'$ to be not local. For every $i=1, \ldots, e$ we have $A_i = A_i'+ (i-1)\e$. \end{theorem}

%show that, also in the case $S'$ is not local, $A_i = A_i'+ (i-1)\e$. 
%We observe the following property of $S$.

Before the proof, we need to discuss some other result. First, since we are dealing also with numerical value semigroups of plane branches, we recall their properties (see for instance \cite[Definition 1.3]{DMS}).

\begin{oss}
\label{planebranch}
Let $S$ be a numerical semigroup minimally generated by $g_1, \ldots, g_n$. For $i=2, \ldots, n$ define $\tau_i$ to be the minimal positive integer $h$ such that $(h+1)g_i \in \langle g_1, \ldots, g_{i-1} \rangle$. Let $\Ap(S)$ be the Ap\'{e}ry set of $S$ with respect to the minimal nonzero element $e= g_1$.

The semigroup $S$ is the value semigroup of a plane branch if $$\Ap(S)= \lbrace \sum_{i=2}^n \lambda_ig_i | 0 \leq \lambda_i \leq \tau_i   \rbrace$$ and $(\tau_i+1)g_i < g_{i+1}$ for every $i=2, \ldots, n-1$.
It can be observed that if $S$ is the value semigroup of a plane branch and $\Ap(S)=\{\omega_1,\ldots,\omega_{e}\}$, then $\omega_i > (i-1)e$ and $\omega_i- \omega_{i-1} > e$. 
\end{oss}

 By \cite[Proposition 2.2.c]{apery} the fact that $S'$ is not local is equivalent to say that $\Delta^S(\e) \neq \emptyset$. This implies the following property.

\begin{lemma}
\label{lemminofacile}
Let $S \subseteq \mathbb{N}^2$ be a symmetric local good semigroup such that $\Delta^S(\e) \neq \emptyset$. Then all the absolute elements of $S$ are in $A:=\Ap(S)$.
\end{lemma}

\begin{proof}
Since $S$ is local and symmetric, $\g \in S$. Symmetry together with the fact that $\Delta^S(\e) \neq \emptyset$ implies that $\g \in A$. Now if $\al$ is an absolute element of $S$, by symmetry %(see \cite[Proposition 4.3]{DGM}) 
$\g - \al \in S$ and is another absolute element. Since $\al + (\g - \al) = \g \in A$, necessarily $\al \in A$.
\end{proof}

Next proposition describes how the absolute elements of $S$ are related to elements in $S'$.
Define the following elements of $S'$. For $j= 1, \ldots, e_1$, $k=1, \ldots, e_2$, let
$$ \om_{j,k} := (u_j-(j-1)e_1, v_k-(k-1)e_2). $$
We show that these elements come from the absolute elements of $S$ after blowup.

\begin{prop}
\label{absoluteelements}
Let $S=v(\mathcal{O}) \subseteq S_1 \times S_2 \subseteq \mathbb{N}^2$ be the local value semigroup of a plane curve and suppose $S'= S_1'\times S_2'$ to be not local. Let $i \in \{1, \ldots, e \}$ Then: 
\begin{itemize}
\item The absolute elements of the level $A_i$ of $\Ap(S)$ are the elements $ \om_{j,k} + (i-1)\e  $ such that $j+k-1=i$.
\item If $i> 1$ and $\al$ is an absolute element of $S$ in the level $A_{i-1}$, then $\Delta^S(\al + \e)$ contains some absolute element of $S$ in the level $A_i$.
\end{itemize}
% For every $i=1, \ldots, e$, the absolute elements of the level $A_i$ of $\Ap(S)$ are the elements $ \om_{j,k} + (i-1)\e  $ such that $j+k-1=i$. Moreover, if $\al$ is an absolute element of $S$ in the level $A_{i-1}$, then $\Delta^S(\al + \e)$ contains absolute elements of $S$ in the level $A_i$.
\end{prop}

\begin{proof}
Since $ \om_{1,1} + (1-1)\e = \boldsymbol{0} \in A_1 $, by induction we can assume the thesis true for all the levels indexed by numbers smaller than $i$. Take $j,k$ such that 
$j+k-1=i \geq 2$ and call $\be:= \om_{j,k}+ (i-1)\e$. 
First consider the case $j=1, k \geq 2$, see Figure \ref{Figura 6} (the case $k=1, j \geq 2$ is analogous). By inductive hypothesis $ \om_{1,k-1} + (i-2)\e  $ is an absolute element of $S$ in the level $A_{i-1}$. Hence $ \de:= \om_{1,k-1} + (i-1)\e  $ is an absolute element of $E$ and $\be \in \Delta_1(\de) $.

\begin{figure}[H]
 \begin{minipage}{.5\textwidth}
\includegraphics[scale=0.5]{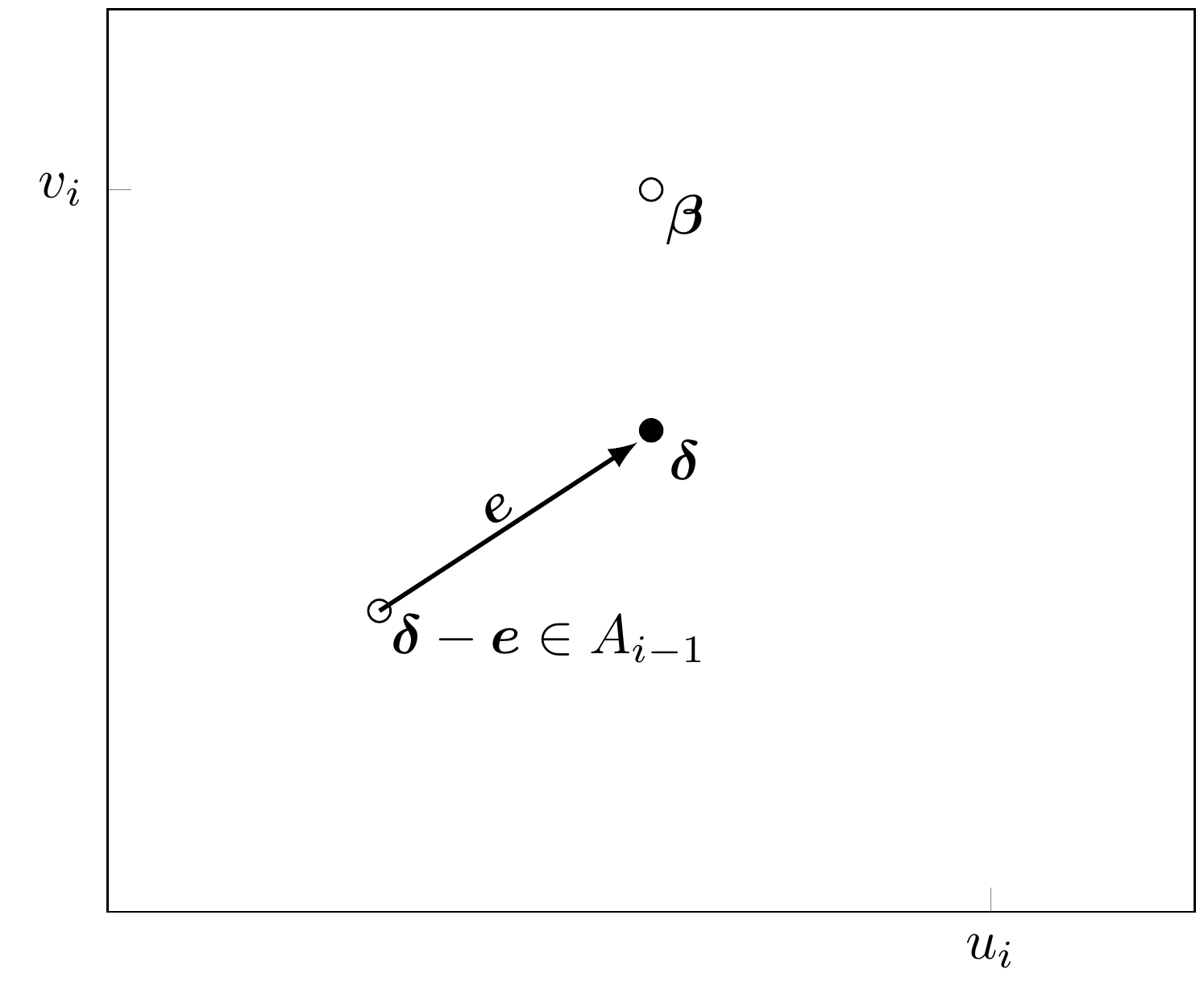}
\caption{\footnotesize In this figure $\bs{\delta}-\bs{e}$ is an absolute element\newline
in $A_{i-1}$, $\bs{\delta}$ is an absolute element in $S+\bs{e}$}
\label{Figura 6}
\end{minipage}
 \begin{minipage}{.5\textwidth}
\includegraphics[scale=0.5]{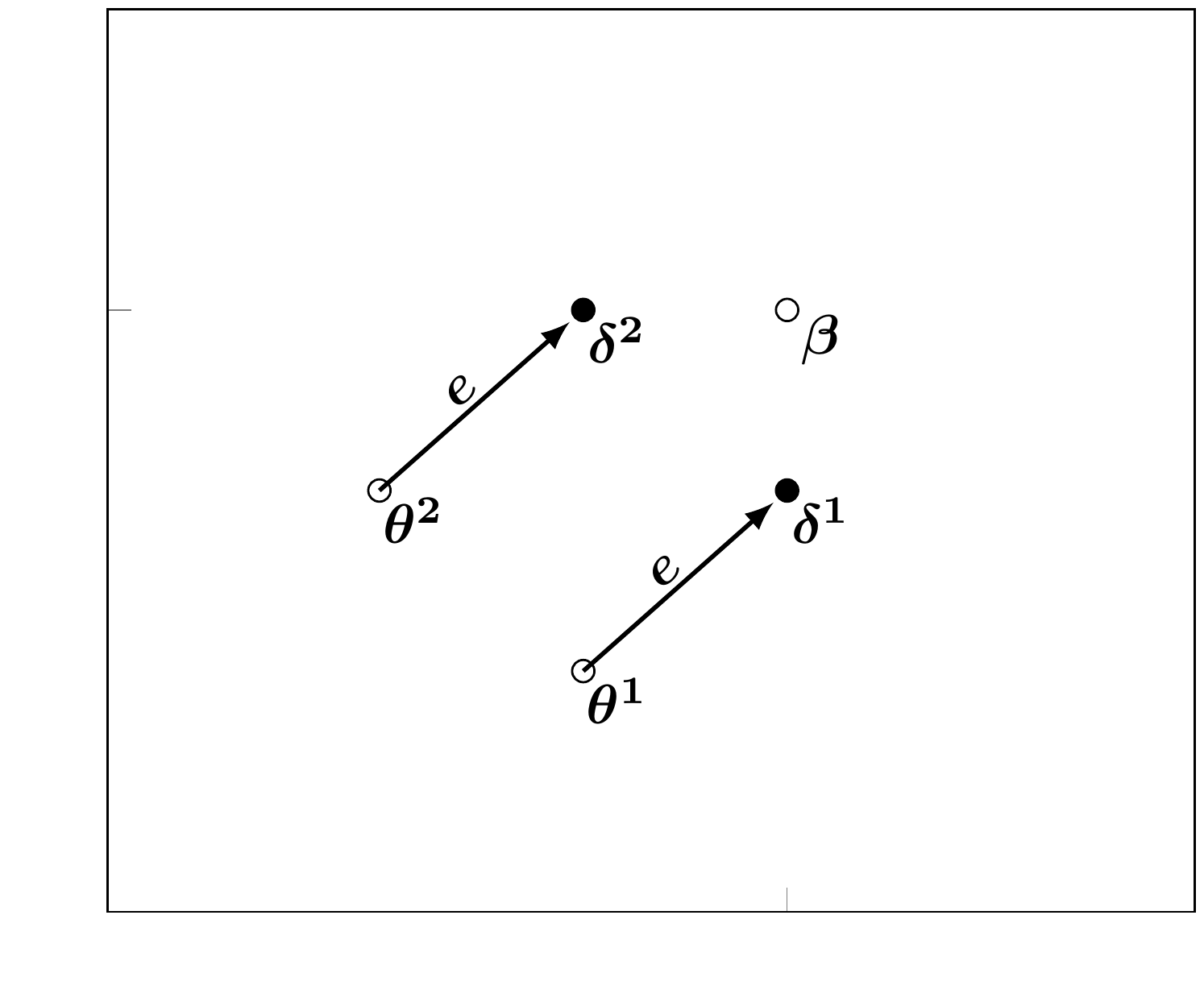}
\caption{\footnotesize In this figure $\bs{\theta}^1,\bs{\theta}^2$ are absolute elements in $A_{i-1}$, $\bs{\delta}^1,\bs{\delta}^2$ are absolute elements in $S+\bs{e}$}
\label{Figura 7}
\end{minipage}
 \end{figure}

By Lemma \ref{lemminofacile}, an absolute element of $E$ cannot be an absolute element of $S$, therefore $\Delta^S(\de)$ is non-empty and contained in $A$. In particular, by Corollary \ref{wellbehavedcor} it is all contained in the same level, say $A_h$ with $h \geq i$. Observe that in this case, since $i=k$, $$\be = (u_1+(i-1)e_1, v_k-(k-1)e_2+(i-1)e_2)= ((i-1)e_1, v_i) \in \Delta_2(0, v_i).$$
By Corollary \ref{wellbehavedproj}, $\Delta^S_2(0, v_i) \subseteq A_i$ and it is non-empty. Inductively, we can assume also that $\Delta^S_2(0, v_p)$ contains only one element for every $p < i$ (this is obviously true for $p=1$). We aim to show that the only element in $\Delta^S_2(0, v_i)$ is $\be,$ proving that indeed it is an absolute element of $S$ in $A_i$. 
%Since $A$ is well-behaved we must have $\Delta^S(\be)= \emptyset$....

Suppose there exists a maximal (thus absolute) element $\he^1 \in \Delta^S_2(0, v_i)$, $\he^1 < \be $. Hence, since $S_1$ is a plane branch semigroup, $\eta^1_1 < (i-1)e_1 < u_i$. Now, again by Corollary \ref{wellbehavedproj}, also the set $\Delta_1^S(u_i,0) \subseteq A_i$ and it is non-empty. Therefore it must contain an absolute element $\he^2$, otherwise we would find elements in $A_i$ dominating $\he^1$. Recalling that $A$ is well-behaved and using Lemma \ref{2ramiwb}, necessarily $ \he^1 \wedge \he^2 \in E$. We can then find an element $\he \in E$, such that $\Delta^S(\he) \subseteq A$, and $\he^1 \in \Delta^S_1(\he)$ (such element exists by property (G2) of $E$). 
It follows that $\he - \e$ is an absolute element of $S$. By Lemma \ref{lemminofacile}, we can say that $\he - \e \in A_p$ with $p < i$. 
Since, by induction, $ \de - \e = \om_{1,k-1} + (i-2)\e  $ is the only element in $\Delta^S_2(0, v_{i-1})$, necessarily % $\om_{1,k-1} + (i-2)\e \gg \he - \e$. 
$\he- \e \ll \de - \e$. 
Indeed, the first coordinate is strictly smaller by assumption and the second one must be smaller, otherwise property (G1) would contradict the uniqueness of $\de - \e$ on its horizontal line.

Hence $p< i-1$. If $i=2$ this is impossible. Otherwise, by inductive hypothesis on the second statement of the theorem, this implies that $\Delta^S(\he)$ contains some absolute element in the level $p+1$.
 This is also impossible since $\he^1 \in \Delta^S_1(\he) \cap A_i$. Hence we cannot have $\he^1 < \be$.
 
 Suppose either $\he^1 > \be $ or there are infinitely many elements in $\Delta^S_2(0, v_i)$. If $\be \in S$, then automatically $\be \in A$ and we get a contradiction since $\Delta^S(\be) \subseteq \Delta^S_1(\de) \cup \Delta^S_2(0, v_i) \subseteq A$ and $A$ is well-behaved.
 %$\be \in \Delta_1(\de) \subseteq A$ and $A$ is well-behaved. 
 If $\be \not \in S$, there would exist an element in $\Delta_2(0, v_i) \cap A_i$ dominating the maximal element in $\Delta_1(\de)$, which is in the level $A_h$ with $h \geq i$. Again a contradiction. This proves $\be \in A_i$, $\Delta^S(\be)= \emptyset$ and $\be$ is the only element in $\Delta^S_2(0, v_i)$.

%\bigskip

We deal now with the case $j,k \geq 2$ (see Figure \ref{Figura 7}). Similarly as before, we start by the inductive hypothesis that $\te^1:= \om_{j,k-1} + (i-2)\e  $ and $\te^2:= \om_{j-1,k} + (i-2)\e  $ are absolute elements of $S$ in the level $A_{i-1}$. Hence $ \de^1:= \om_{j,k-1} + (i-1)\e  $ and $ \de^2:= \om_{j-1,k} + (i-1)\e  $ are absolute elements of $E$ and $\be \in \Delta_1(\de^1) \cap \Delta_2(\de^2).$
Again $\Delta^S_1(\de^1),\Delta^S_2(\de^2) \neq \emptyset$ and they are contained in $A$. This, together with $A$ well-behaved, %Lemma \ref{wellbehavedlem1}, 
shows that, if $\be \in S$, then it is absolute.

Suppose now that $i= j+k-1 \leq e_2$ (or analogously $i \leq e_1$). Starting by
$\om_{1,i}+ (i-1)\e$, which has been treated previously, we can suppose working by induction on $j$ that we already proved that $\om_{j-1,k+1}+ (i-1)\e$ is an absolute element and it is in $A_i$. Since $\om_{j-1,k+1}+ (i-1)\e \in \Delta^S_1(\de^2)$, by %Lemma \ref{wellbehavedlem1} and 
Corollary \ref{wellbehavedcor}, we get $\Delta^S_2(\de^2) \subseteq A_i$. Hence, if $\be \in S$, then it is in $A_i$.
%The first thing to prove now is that $\be \in S$ and it is an absolute element. 
%Since $\Delta^S(\be) \subseteq A$, by Lemma \ref{wellbehavedlem1} it is all contained in the same level. Hence also $\Delta^S_1(\de^1),\Delta^S_2(\de^2)$ are contained in that same level. It follows that the elements of these two sets are not comparable with respect to the order $\ll$.
%Thus, assuming $\Delta^S(\be) \neq \emptyset $, we would contradict the fact that $A$ is well-behaved.
%For $l=1,2$, let $\he^l$ be the maximal element of $\Delta^S_l(\de^l)$. We know that $\he^l \leq \be$ and proving $\be \in S$ it is equivalent to prove $\he^1=\he^2=\be$. Say that $\he^1 \neq \he^2$. Since $A$ is well-behaved, $\he:= \he^1 \wedge \he^2 \in E$.

As before, suppose $\be \not \in S$ and there exists an absolute element $\he^1 \in \Delta^S_2(\de^2) \subseteq A_i$ and $\he^1 < \be$. If also $e_1 \geq i$, by the previous case $\he^2 := \om_{i,1}+ (i-1)\e \in A_i$ is an absolute element and $\eta^2_1 > \beta_1$.
%inductively we can assume also $\om_{j+1,k-1}+ (i-1)\e \in A_i$ and therefore we find an absolute element $\he^2 < \be$, $\he^2 \in \Delta^S_1(\de^1) \subseteq A_i$.   
If instead $i > e_1$, there exists an element $\he^2 \in A_i$ such that $\Delta_2(\he^2)$ is an infinite horizontal line in $A_i$ (for this see \cite[Lemma 1, items 7-8 and Theorem 5]{DGM}). In both case $\he^1 \wedge \he^2 \in E$ and, as in the previous case, we can find an absolute element of $E$, called $\he$, such that $\he^1 \in \Delta^S_1(\he)$. This yields to a contradiction. Indeed $\he - \e = (s_1, s_2)$ is an absolute element of $S$ in a level $A_p$ with $p < i$. But, exactly as in the previous case we cannot have $p< i-1$. Thus $p=i-1$, but $\theta^2_1 < s_1 <\theta^1_1$, $s_2 < \theta^2_2$, and this contradicts the inductive hypothesis describing all the absolute elements of $A_{i-1}$. Concluding as in the case $j=1$, we get $\be \in A_i$ is absolute.

The remaining case has the assumption $i > e_1,e_2$. Following the same process above, if we could say that $\Delta^S_2(\de^2) \subseteq A_i$, we would be able to conclude exactly in the same way. Call again $\he^1$ the maximal element in $\Delta^S_2(\de^2)$. Say that $\he^1 \in A_h$ with $h \geq i$. We use now the duality property of the levels of $A$ (see \cite[Proposition 3.10]{apery} or use Proposition \ref{dualityprop2} together with $A$ well-behaved). %Recall that $e= e_1+e_2$. 
Thus, $\g + \e - \de^2 \in \Delta^S(\g +\e - \he^1) \subseteq A_{e-h+1}$ and it is an absolute element, since $\de^2$ is absolute in $E$ and $S$ is symmetric.
Hence $(\g + \e - \de^2) + \e$ is an absolute element of $E$. Since $e-h+1 \leq e-i+1< \min\{e_1+1, e_2+1\}$, by \cite[Theorem 5]{DGM} $A_{e-h+1}$ is a finite level. All the finite levels can be considered in the previous cases, therefore $\Delta^S(\g + 2\e - \de^2) \subseteq A_{e-h+2}$. But at the same time $\g + 2\e - \de^2= \g + \e - (\de^2 - \e)$ and $\de^2- \e $ is in $A_{i-1}$. Thus $\Delta^S(\g + 2\e - \de^2) \subseteq A_{e-i+2}$ implying $h=i$.

The last thing to do is proving by induction that the elements $ \om_{j,k} + (i-1)\e  $ such that $j+k-1=i$ are the only absolute elements in the level $A_i$. Say that $\al$ is another absolute element in the level $A_i$. If there exist $\te^1= \om_{j,k}+(i-1)\e$ and $\te^2= \om_{j-1,k+1}+(i-1)\e$ such that $\al \gg (\te^1 \wedge \te^2)$ we easily get a contradiction using property (G1) since, following our construction, $\Delta^S(\te^1 \wedge \te^2) \subseteq A_i$ and $A$ is well-behaved.
The other possible situation arises if $\al$ is minimal among the absolute elements of $A_i$ with respect to one coordinate, say the first one. %By Lemma \ref{wellbehavedproj}, this can happen only if $i > e_2$ and there exist $\de \in E$ such that....
Let $\te$ be the absolute element of type $\om_{j,k}+(i-1)\e$ minimal with respect to the first coordinate. By construction $\te \in \Delta^S_1(\de)$ where $\de= \te' + \e$ and $\te'$ is the absolute element in the level $A_{i-1}$ minimal with respect to the first coordinate. But $\te \wedge \al \in E$, and thus there exists an element $\he$, absolute in $E$, such that $\al \in \Delta^S_1(\he)$. It follows that $\he - \e$ is an absolute element and it is in some level $A_p $ with $p < i$. If $p< i-1$, this contradicts the inductive hypothesis as before, while if 
$p=i-1$, it contradicts the minimality of $\te'$. 
\end{proof}

\begin{oss}
\label{seguedaproof} \rm
Assume the same setting and notation of Theorem \ref{teoremaaperynonlocale}.
Let $\al \in A_i$ such that $\al \ll \g + \e$ (obviously $i < e$). By Proposition \ref{structurelevel}, $\al $ shares a coordinate with an absolute element in $A_i$. By what observed in the proof of Proposition \ref{absoluteelements}, the cases (iv)-(v) of Proposition \ref{structurelevel} can be excluded since for $i \leq e_2$, $j \leq e_1$, $\Delta^S_2(0, v_i)$ and $\Delta^S_1(u_j,0)$ contains only one element. 
Hence, using the notation of Proposition \ref{structurelevel}, $\al \in \Delta^S(\de)$ where $\de$ is either the minimum of two "consecutive" absolute elements $\te^{(k)}, \te^{(k+1)}\in A_i$ or the minimum of an absolute element in $A_i$ and of the elements of an infinite line in $A_i$. Moreover, $\de$ is an absolute element of $E$.
\end{oss}

We are ready to prove Theorem \ref{teoremaaperynonlocale}.

%\begin{theorem}
%\label{teoremaaperynonlocale}
%Let $S=v(\mathcal{O}) \subseteq S_1 \times S_2$ be a local value semigroup of a plane curve and suppose $S'= S_1'\times S_2'$ to be not local. For every $i=1, \ldots, e$, $A_i = A_i'+ (i-1)\e$. \end{theorem}

\begin{proof}(of Theorem \ref{teoremaaperynonlocale})
For this proof we will use the structure of Ap\'ery set of a non-local good semigroup described in Theorem \ref{thmnonlocallevels} and Corollary \ref{aperylocale}.
Proposition \ref{absoluteelements} describes the absolute elements of $S$ in the level $A_i$ as exactly the opportune translation of the elements $\al$ in $\Ap(S')$ such that $\al \in A_i$ and $\Delta^S(\al) \subseteq \Ap(S)$. Equivalently such $\al$'s are the elements $(a_1, a_2)$ such that $a_1 \in \Ap(S_1',e_1)$ and $a_2 \in \Ap(S_2',e_2)$. 

Notice that by Lemma \ref{lemminofacile}, $\g \in \Ap(S)$. If $\al \neq \g$ is an absolute element, then $\g \gg \al$. Hence $\g $ is in a strictly larger level than all the other absolute elements of $S$. Clearly $A_e = \Delta(\g+\e)$ does not contain any absolute element. By Proposition \ref{absoluteelements}, the level $A_{e-1}$ contains only one absolute element, namely $\om_{e_1, e_2}+ (e-2)$. Thus $\g= \om_{e_1, e_2}+ (e-2)$.
%For simplicity we assume $e_1, e_2 \geq 2$. If one of them is equal to $1$, the same argument can be replied in a simplified way.
%Notice that the level $A_2$ is then finite and has only two absolute elements (namely $\om_{1,2} + \e$, $\om_{2,1} + \e$) and they both are in $\Delta^S(\e)$. Hence $A_2 = \Delta^S(\e)$. The symmetry of $S$ and \cite[Proposition 3.10]{apery} imply that $\g= (\g + \e) - \e$ is an absolute element in the level $A_{e-1}$. Hence by Proposition \ref{absoluteelements}, $\g= \om_{e_1, e_2}+ (e-2)$. \bf discutere il caso $e_1=1$\rm

Now we are able to show that the infinite lines in $\Ap(S)$ come from the infinite lines in $\Ap(S')$ according to the formula $A_i = A_i'+ (i-1)\e$. We do this for the vertical lines, but the argument for the horizontal is analogous switching the components. Using \cite[Lemma 1 items 7-8 and Theorem 5]{DGM} it is sufficient to consider the elements $\al^i=(s_i, \gamma_2+e_2+1) \in A_i$ for $i=e_2+1, \ldots, e$. Set $j:=i-e_2$. Since the coordinates of the infinite vertical lines in $\Ap(S')$ correspond to the elements of $\Ap(S_1')$, we have to prove that
$$ s_i= (u_j- (j-1) e_1) + (i-1)e_1 = u_j + e_1e_2. $$ By \cite[Theorem 8]{DGM}, the (first) coordinates of infinite vertical lines in $\Ap(S)$ are dual with respect to $\gamma_1 + e_1$ to the coordinates of $\Ap(S_1)$. Thus we know that $s_i = \gamma_1 + e_1 - u_{e_1-j+1}$. But since $\g = \om_{e_1, e_2}+ (e-2)\e $, we get $$ \gamma_1 = (u_{e_1}-(e_1-1)e_1)+ (e_1+e_2-2)e_1 = u_{e_1}-e_1 + e_1e_2. $$ To conclude we just need to recall that, since $S_1$ is symmetric, $u_{e_1} = u_j + u_{e_1-j+1}$.

Finally, it remains to prove that $A_i = A_i'+ (i-1)\e$ only for elements $\ll \g + \e$ that are not absolute (the behavior of the absolute elements follows by Proposition \ref{absoluteelements}). Since the level $A_e$ does not contain any element of this kind, we can assume inductively that for all the levels $A_p$ with $p > i$ the thesis is true.

Let $\al=(a_1,a_2) \in A_i$. By Proposition \ref{structurelevel}, $\al $ shares a coordinate with an absolute element in $A_i$. %By what observed in the proof of Proposition \ref{absoluteelements}, the cases (iv)-(v) of Proposition \ref{structurelevel} can be excluded since for $i \leq e_2$, $\Delta^S(0, v_i)$ contains only one element. 
%Hence, $\al \in \Delta^S(\de)$ where $\de$ is either the minimum of two "consecutive" absolute elements in $A_i$ or the minimum of an absolute element in $A_i$ and of the elements of an infinite line in $A_i$. Moreover, $\de$ is an absolute element of $E$.
By Remark \ref{seguedaproof}, suppose that $\al \in \Delta^S(\de)$ where $\de$ is an absolute element of $E$. % is either the minimum of two "consecutive" absolute elements in $A_i$ or the minimum of an absolute element in $A_i$ and of the elements of an infinite line in $A_i$.

Say that $\al \in \Delta^S_2(\de)$. Thus, again by Remark \ref{seguedaproof}, $\Delta^S_2(\al)$ contains either an absolute element or an infinite line in $A_i$. In both cases it follows by the previous part of the proof that
$a_2 - (i-1)e_2 \in S_2'$.
We want to show that also $a_1 - (i-1)e_1 \in S_1'$.
If $\Delta^S_1(\al)$ contains some element of $A$, then it must contain some element in $A_h$ with $h > i$.
By inductive assumption, the thesis is true for the level $A_h$ and we get $a_1 - (i-1)e_1 = a_1 - (h-1)e_1 + (h-i)e_1 \in S_1'$. 

%Furthermore, if $\Delta^S_1(\al)$ contains either an absolute or an infinite line in $A_h$, necessarily $h > i$ and $a_1 - (i-1)e_1 = a_1 - (h-1)e_1 + (h-i)e_1 \in S_1'$. 
Otherwise, $\Delta^S_1(\al)$ contains only elements of $E$ which eventually form an infinite line by Lemma \ref{lemminofacile}. Hence, $a_1 = s_p+me_1$, where $m \geq 1$ and $s_p=u_{p-e_2} + e_1e_2$ is the coordinate in $S_1$ of the infinite vertical line of $S$ contained in the level $A_p$ for some $p$. 

Suppose $p \leq i$. This implies that the level $A_i$ contains an infinite vertical line of coordinate $s_i=u_{i-e_2} + e_1e_2$.
By Remark \ref{planebranch}, since $S_1$ is a plane branch, for every $h > 1$, $u_h - u_{h-1} > e_1$. Clearly $s_i < a_1$, thus $$ 0 <  a_1-s_i = - (u_{i-e_2} - u_{p-e_2}) +me_1 < (-i+p+m)e_1.  $$
Therefore $p> i-m$. Again this shows $a_1 - (i-1)e_1 = a_1 - (p+m-1)e_1 + (p+m-i)e_1 \in S_1'$.

Suppose now $p > i$. In this case we get a contradiction finding an element in $A_p \cap \Delta^S_1(\al)$. By Lemma \ref{2ramiwb}, there exists $\be \in A_{p}$ such that $\be \gg \al$. Since $s_p < a_1$, using the notation of Proposition \ref{structurelevel}, there exist two elements $\te^{(t)}, \te^{(t+1)}\in A_p$ with $t \geq 0$, such that $\theta_1^{(t)} < \alpha_1 < \theta_1^{(t+1)}$. The element $\om:= (a_1, \theta_2^{(t+1)}) \in \Delta^S_1(\al) \cap \Delta^S(\te^{(t)} \wedge \te^{(t+1)})$. By Remark \ref{seguedaproof}, $\om \in A_p$. 

In all the above cases $\al - (i-1) \e \in S_1' \times S_2' = S'$.
Combining now Theorem \ref{aperylocale}, Proposition \ref{absoluteelements} and the definition of $\om_{j,k}$ it follows that $\al - (i-1) \e \in A_i'$. 

To prove the opposite inclusion, let $\be \in A_i'$ such that $\be \ll \g(S')+\e$. If $\be = \om_{j,k}$ with $j+k-1=i$, clearly we are done by Proposition \ref{absoluteelements}. Hence assume that $\be=(b_1,b_2) < \om_{j,k}$ and $\be \in \Delta^S_2(\om_{j-1,k})$ with $j+k-1=i$ (or similarly in $\Delta^S_1(\om_{j,k-1})$). By Proposition \ref{absoluteelements}, if $\be + (i-1) \e \in S$, then it must belong to the level $A_i$. It suffices then to show $\be + (i-1) \e \in S$. 

Suppose not.
Since $b_1 \in S'_1$ and $u_{j-1} - (j-2)e_1 < b_1 < u_j-(j-1)e_1$, then $b_1= (u_{h}-(h-1)e_1) + me_1$ with $h \leq j-1$, $m\geq 1$. Since $i \geq j > h$, 
$b_1 + (i-1)e_1 = b_1 + (h-1)e_1 + (i-h)e_1 \in S_1 $. Thus there exists some element in $S$ having the first coordinate equal to $b_1 + (i-1)e_1$. 
%After noticed this, 
Since $\om_{j,k} + (i-1)\e \in \Delta^S_2(\be + (i-1) \e) $, there must exist
 an absolute element $\te=(t_1, t_2) \in S$ with $t_1= b_1+(i-1)e_1 $ and $t_2 < b_2+(i-1)e_2$ (otherwise, property (G1) would contradict the assumption $\be + (i-1) \e \not \in S$). %By Lemma \ref{lemminofacile}, $\te \in \Ap(S)$ and, 
 Since $\te \ll \om_{j,k}+(i-1) \e$, then $\te \in A_p$ with $p<i$. 
%Thus $\te - (p-1) \e \in \Ap(S')$ and  
Thus $t_1 -(p-1)e_1 \in \Ap(S'_1, e_1)$. %for $l=1,2$. 
But $t_1 -(p-1)e_1 = b_1+(i-1)e_1-(p-1)e_1 = b_1 + (i-p)e_1 \not \in \Ap(S'_1, e_1) $ since $p < i$. This is a contradiction and shows $\be + (i-1) \e \in S$. % and concludes the proof of this theorem.
\end{proof}

\section*{Acknowledgements}
 The first author is supported by the NAWA Foundation grant Powroty "Applications of Lie algebras to Commutative Algebra" - PPN/PPO/2018/1/00013/U/00001.
The other two authors are funded by the project PIA.CE.RI 2020-2022 Università di Catania - Linea 2 - "Proprietà locali e globali di anelli e di varietà algebriche". The authors wish to thank Marco D'Anna for the interesting and helpful discussions about the content of this article. % and to thank the referee for the careful revision and the comments that improved the quality of this article.

 %\bf togliere quello che non serve dalla bibliografia \rm

\end{document}